\documentclass[12pt,a4paper,twoside,reqno]{amsart} 
\usepackage{amsfonts, amsthm, amsmath, amssymb}
\usepackage{hyperref}
\hypersetup{colorlinks=false}

\usepackage[margin=2.9cm]{geometry}

\usepackage{helvet}

\usepackage{graphicx}
\usepackage{amsmath}

\usepackage{amsthm}
\newtheorem{theorem}{Theorem}

\newtheorem{lemma}{Lemma}[section]

\newtheorem{remark}{Remark}

\begin{document}
	\author{Sumit Kumar, Kummari Mallesham and Saurabh Kumar Singh}
	\title{Sub-convexity bound  for  $GL(3) \times GL(2)$ $L$-functions: the depth aspect}
	
	%

	\address{Sumit Kumar \newline  Stat-Math Unit, Indian Statistical Institute, 203 B.T. Road, Kolkata 700108, India; email: sumitve95@gmail.com}
	
	\address{ Kummari Mallesham \newline Stat-Math Unit, Indian Statistical Institute, 203 B.T. Road, Kolkata 700108, India;  email:iitm.mallesham@gmail.com
	}
	
	\address{ Saurabh Kumar Singh \newline  Department of Mathematics and Statistics, Indian Institute of Technology, Kanpur  208016, India; \newline  Email: saurabs@iitk.ac.in
	} 
	
	\date{\today}
	
	\maketitle
	
	\begin{abstract}
		In this article, we will prove the following  sub-convex bound 
			\begin{align*}
			L \left( {1}/{2}, \pi \times f \times \chi \right) \ll_{f, \pi, \epsilon} (p^r)^{{3}/{2} -3/20+\epsilon},
		\end{align*}
	for $GL(3) \times GL(2)$ Rankin Selberg $L$-functions. 
	\end{abstract}
	\section{Introduction}
	Let $\pi$ be a Hecke-Maass cusp form for $SL(3,\mathbb{Z})$ with  the normalised Fourier coefficients $A(n,k)$ and  $f$ be a holomorphic Hecke cusp form for $SL(2,\mathbb{Z})$  with  the normalised Fourier coefficients  $\lambda_f(n)$. Let $\chi$ be a primitive Dirichlet character modulo $ R:= p^{r}$. Rankin-Selberg $L$-series associated to the  above objects is given by 
	\begin{align}\label{Lseries}
		L \left( s, \pi \times f \times \chi \right) =\mathop{\sum \sum}_{n, \, k=1}^{\infty}\frac{A(n,k)\lambda_f(n)\chi(n)}{(nk^2)^s},
	\end{align}
	in the half plane $\Re s >1$. It is well known that this series extends to an entire function to whole of complex plane $\mathbb{C}$ and satisfies a functional equation relating $s$ with $1-s$. It is an important problem to understand the growth of the $L$-function inside the critical strip. Functional equation and the Phragmen-Lindel{\" o}f principle yield the following convex bound 
	\begin{align*}
		L \left( \frac{1}{2}, \pi \times f \times \chi \right) \ll_{f, \pi, \epsilon}R^{\frac{3}{2}+\epsilon},
	\end{align*}
	where the implied constant depends on the forms $\pi$, $f$ and  $\epsilon>0$. The Lindel{\" o}f hypothesis predicts that such a bound holds with any positive exponent in place of $3/2+\epsilon$. But even breaking the convexity barrier is difficult and has remained open so far. In this article, our aim is to prove the following subconvex bound in the case  when  $R=p^r$ tends to infinity (both $p$ or $r$ may vary). 
	\begin{theorem}\label{r>2 thm}
		Let $\pi$ be a Hecke-Maass cusp form for $SL(3,\mathbb{Z})$, $f$ be a holomorphic Hecke cusp form for $SL(2,\mathbb{Z})$ and $\chi$ be a primitive Dirichlet character of prime power modulus $p^r$, $r\geq 3$. Then we have 
		\begin{align*}
			L \left( \frac{1}{2}, \pi \times f \times \chi \right) \ll_{f, \pi, \epsilon} (p^r)^{{3}/{2} -3/20+\epsilon},
		\end{align*}
	where the implied constant depends only on $f$, $\pi$ and  $\epsilon$ only. 
	\end{theorem}
	
\begin{remark}
	Any unitary Hecke character $\Psi$ on the idele group 
	$$\mathbb{A}_{\mathbb{Q}}^{\times}/ \mathbb{Q}^\times=\mathbb{R}_{+} \times \prod_{p}\mathbb{Z}_p^{\times}$$
	can be decomposed as $\Psi=\vert \, .\,  \vert^{it} \otimes (\otimes_p \psi_p)$, where $t \in \mathbb{R}$ and $\psi_{\mathrm{finite}}=\otimes_p \psi_p$ corresponds to a Dirichlet character $\chi$, say. Then the twisted $L$-function $L(1/2,\pi\times f \times \Psi)$ corresponds to $L(1/2+it,\pi\times f \times \chi)$. Thus the $t$-aspect subconvexity bound for $L(1/2+it,\pi \times f)$ of R. Munshi \cite{munshi12}   corresponds to bounds for the twisted $L$-values $L(1/2,\pi\times f \times \Psi)$, where the Hecke characters $\Psi$ are `supported' only at the prime at infinity.  In the present case, we are considering twists by Hecke characters which are `supported'(ramified) at a fixed finite prime $p$.  
	\end{remark}

In the proof of Theorem \ref{r>2 thm}, we do not really require the cuspidiality of $\pi$ and $f$, but it only uses summations formulas of $GL(3)$ and $GL(2)$ Fourier coefficients $A(n,1)$ and $\lambda_{f}(n)$ respectively. We note that the Voronoi summation formula for $d_3(n)$ resembles that of $A(n,1)$, where $d_3(n)$ is the triple  divisor function which corresponds to the Fourier coefficients of a  $GL(3)$ minimal Eisenstein series $E_{\min}$, see, \cite{DG} . The only difference is that we get a main term in this case, see Lemma  \ref{gl3voronoi for d_3}.  We observe that 
\begin{align}
L(s,	E_{\min} \times f\times \chi) \approx L(s,f\times \chi)^3 . 
\end{align}
Thus our approach also  gives a subconvexity bound for $L(1/2, f \times \chi)$ in the depth aspect, and our bounds are  unifrom in $p$ and $r$. We give a sketch for this proof in the  appendix.

It we take the  Eisenstein series, $E=E(\star,1/2)$ say, for $SL(2,\mathbb{Z})$ in place of the cusp form  $f$, then we have 
\begin{align}
	L(s,	\pi \times E\times  \chi) \approx L(s,\pi \times \chi)^2.
\end{align}
Note that the Fourier coefficients of $E(\star,1/2)$ are the divisor function $d(n)$, for which we have a Voronoi summation formula similar to that of $\lambda_{f}(n)$ except for the main term (see Lemma \ref{voronoi for dn}). This main term does not appear in our analysis (see Remark \ref{on gl2 eisenstein}). Thus in this case, treatment of the remaining terms is very similer to the cusp form $f$ case. Hence we get a subconvexity bound  for $L(1/2,\pi \times \chi)$ in the depth aspect and our bounds are uniform in $p$ and $r$.  \\

R. Munshi  in  \cite{Mu1} and \cite{Mu3} obtained subconvexity bound for $L(1/2, \pi \times \chi)$ when $\pi$ is a self-dual $GL(3)$ form.   In  \cite{Mu1}, he  deals with the case when $r$ is fixed and $p \rightarrow \infty$  while in \cite{Mu3},  he considers the case when $p$ is fixed and  $r \rightarrow \infty$. Recently Q. Sun and Zhao  \cite{Sun} extended Munshi's work  \cite{Mu3}  to any $GL(3)$ form, as $r \mapsto \infty $. Their bounds are not uniform with respect to $p$. Our results extend both the works of Munshi (\cite{Mu1} and \cite{Mu3})  to any $GL(3)$ form, also removes the dependency of $p$ in the result of Sun and Zhao  \cite{Sun}.  

	%
	\begin{remark}
		Using the same method along with the amplification trick (see (\cite{PS}), one can also get a similar  result for $r=2$ also. 
\end{remark}
\subsection{A brief history }
	We will now recall a brief history of the problem. We only focus on the level aspect and the  depth aspect.   For degree one $L$-functions,  $q$-aspect sub-convexity bound was first proved by Burgess \cite{DB} in $1962$ using his ingenious technique of completing short character sums. He proved that 
	\begin{equation} \label{burger for gl2}
		L\left(  s, \chi \right) \ll_{s,\epsilon} q^{3/16 + \varepsilon}, 
	\end{equation}
	for fixed $s$ with $\Re s = 1/2$ and for any $\varepsilon >0$.  After Burgess's result, there was not much progess in the level aspect untill 1990. In 2002, Conrey and Iwaniec \cite{CI},  gave a new method to prove the Weyl strength bound ($1/6+\epsilon$)  for the real characters. Recently I. Petrow and  M. Young proved the  Weyl-exponent subconvex bound for any Dirichlet L-function in  \cite{PY} and \cite{PY2}. Their work is based on the method of Conrey and Iwaniec \cite{CI}. In 2014, Mili{\' c}evi{\' c} \cite{DM} obtained a sub-Weyl subconvex bound  for $\chi$ primitive Dirichlet character modulo $p^r$.
	
	For  $GL(2)$ $L$-functions,   level aspect subconvexity problem was settled by Duke et al. in a series of articles  (\cite{DFI}, \cite{DFI-2}, \cite{DFI-2.1}) using a new form of  circle method and  amplification technique. Further refined results for $GL(2)$ $L$ functions  have been obtained in \cite{BYKO}, \cite{BH}, \cite{CI} and \cite{RM-burgess}.  Extending the above mentioned result of Mili{\' c}evi{\' c} to $GL(2)$ $L$-functions, Blomer and Mili{\' c}evi{\' c} in \cite{BM} obtained 
	\[
	L( 1/2+it,f \otimes \chi) \ll_{f,\,  \varepsilon} (1+|t|)^{5/2} p^{7/6} q^{1/3 + \varepsilon},
	\] 
	where $f$ is a holomorphic or Maass newform for $SL(2, \mathbb{Z})$, and $\chi$ is a primitive character of conductor $q=p^r$, with $p$ an odd prime. Note that the  above exponent tends to the  Weyl exponent as $r\rightarrow \infty$. Using the  conductor lower trick  introduced by R. Munshi in \cite{RM3}, S. Singh and R. Munshi  \cite{MS}, obtained the following subconvex bound when $\chi$ is a primitive Dirichlet character of modulus $p^r$ and  $r \equiv 0 \mod 3$. \[
	L\left(1/2+it,f\otimes \chi\right) \ll_{f,t, \varepsilon} p^{\frac{r}{3} +\varepsilon}.
	\] 
	
	In the case of degree three $L$-functions, the first sub-convex bound  was obtained by V. Blomer \cite{VB} for the self-dual forms.  He obtained the bound when $\chi$ is a quadratic character with prime modulus. For any primitive Dirichlet character of prime power modulus $p^r$, the  subconvex estimates were proved by R. Munshi (\cite{Mu1} and \cite{Mu3}). For any genuine $GL(3)$ $L$-functions, the  subconvexity problem was settled by R. Munshi in \cite{RM2} and \cite{RM4}. In  \cite{RM2} and \cite{RM4}, he considered the cases of moduli which are  product of primes $M_1M_2$  with $M_2^{1/2}<M_1<M_2$,  and  prime respectively. Also for any degree three $L$-functions, Q. Sun  and R. Zhao proved subconvex bounds in the depth aspect in \cite{Sun}. 
	
	For  $GL(3) \times GL(2)$ $L$-functions, the first result in this direction was obtained by V. Blomer (\cite{VB}). He proved the subconvex bound for $L(s, \pi \times f\times \chi)$, where $\pi $ is a self dual form, $f$ is a $GL(2)$ form and $\chi$ is a  quadratic character of prime modulus. This result was generalised to any $GL(3)$ form by P. Sharma (\cite{PS}) recently. 
	
\subsection{Comment on the method}
Our problem is arithmetic in nature. Using the functional equation (see Lemma \ref{AF}), our proof boils down to getting cancellations in the following sum:
	\begin{align}
	 \mathop{\sum}_{n \sim p^{3r}} \, A(n,1) \, \lambda_{f}(n)\chi(n).
\end{align}
Note that the `arithmetic conductor' of $L(1/2,\pi\times f\times \chi)$ is $p^{6r}$. We compare the above sum with the following sum
	\begin{align}
	\mathop{\sum}_{n \sim t^{3}} \, A(n,1) \, \lambda_{f}(n)n^{it},
\end{align}
to which R. Munshi  \cite{munshi12} was seeking  cancellations to get the $t$-aspect subconvexity bound for $L(1/2+it, \pi\times f)$. His approach was to apply the  circle method to separate the oscillations $A(n,1)$ and $\lambda_{f}(n)n^{it}$. While doing so, he introduced an integral which helps to lower the `conductor'. In  analytic problems, this trick is not absolutely necessary and can be removed (see \cite{LS}).  After separating the oscillations, he employs the summations formulas to dualize the sums. While analyzing the resulting character sum, he observes that the character sum boils down to an additive character,  which also plays a crucial role in our proof as well as in  \cite{Sumit}, \cite{KMS } and \cite{PS}.   In our case, following Munshi \cite{munshi12}, we also separate the oscillations  $A(n,1)$ and $\lambda_{f}(n)\chi(n)$. To lower the conductor, we introduce a congruence equation modulo $p^\ell$ (it does not change the conductor of the $L$-function due to the presence of $\chi$),   where $\ell$ is some parameter $\ell <r$. It turns out that this congruence equation trick is very crucial in our approach, without which the circle method approach will not work. Our aim in this paper is to show that the circle method approach works equally well to  give subconvexity bounds in the depth aspect with the same quality of bounds as  in the $t$-aspect.  \\

We now compare our approach with P. Sharma's recent result \cite{PS}, in which he considers the twist aspect ($\chi$ modulo $p$, $p$ varying). He also follows Munshi's method \cite{munshi12} and  also uses the  congruence equation trick. In his case, he had to transfer some `mass' from the $GL(3)$-coefficients to get required savings in  diagonal terms.  In our case, we do not need to transfer `mass' from the $GL(3)$-coefficients, as $\chi$ is a character modulo $p^r$, $r\geq 3$. Since, both  problems  are arithmetic, we end up with a character sum  in which we seek square root cancellations. P. Sharma \cite{PS} had to appeal to Deligne's bound (Riemann hypothesis for varieties over finite fields) to show square root cancellations. In our case, we don't require Deligne's bound. We achieve the required cancellations by elimentary means, albeit tedious.  
 

\subsection*{Notations} In this paper, the notation $\alpha \ll A$ will mean that for any $\epsilon>0$, there is a constant $c$ such that $|\alpha| \leq cA(p^r)^{\epsilon}$. $A \sim B$ will also have the standard meaning, i.e., $B\leq A \leq 2B$. Also, $A \asymp B$ will mean that $(p^r)^{-\epsilon}c_1A \leq B \leq Ac_2(p^r)^\epsilon$, for some absolute constants $c_1$ and $c_2$. We follow the standard $\epsilon$ convention, i.e., $\epsilon$ may vary from places to places. 
	
\subsection*{Acknowledgements} The authors are grateful to Prof. Ritabrata Munshi for sharing his beautiful ideas, explaining his ingenious method in full detail, and his kind support throughout the work. They  would also like  to thank Prof. Satadal Ganguly  for their encouragement and  constant support and  Stat-Math Unit, Indian Statistical Institute, Kolkata for the excellent research environment. Finally, authors would  like to thank the referee for his/her suggestions and comments which really helped to improve the  presentation of the article. 
	\section{\bf Preliminaries}
	In this section, we will recall some known results which we need in the proof. 
	\subsection{ Holomorphic forms on $GL(2)$}
	Let $f$ be a holomorphic Hecke eigenform  of weight $k_{f}$ for the full modular group $SL(2,\mathbb{Z})$. The Fourier expansion of $f$ at $\infty$ is given by 
	$$f(z) = \sum_{n=1}^{\infty} \, \lambda_{f}(n) \, n^{(k_{f}-1)/2} \, e(nz),$$
	for $z \in \mathbb{H}$. We have a well-known  Deligne's bound for the Fourier coefficients which says that  
	\begin{align}\label{raman bound for gl2}
		\vert \lambda_{f}(n) \vert \leq d(n),
	\end{align}
	for  $n\geq 1$, where $d(n)$ is the divisor function. We now state the Voronoi summation formula for $f$ in the following lemma. 
	
	\begin{lemma} \label{gl2voronoi}
		Let $\lambda_{f}(n)$ be as above and $g$ be a smooth, compactly supported function on $(0, \infty)$. Let $a$, $q \in \mathbb{Z}$ with $(a,q)=1$. Then we have
		$$\sum_{n=1}^{\infty} \lambda_{f}(n) \, e\left(\frac{an}{q}\right)g(n) = \frac{2\pi i^{k_f}}{q} \sum_{n=1}^{\infty} \lambda_{f}(n) \, e\left(-\frac{d n}{q}\right)\, h(n),$$
		where $ad \equiv 1 (\mathrm{mod } \, q)$ and 
		$$h(y) = \int_{0}^{\infty} g(x) \, J_{k_{f}-1} \left(\frac{4 \pi \sqrt{xy}}{q}\right) \, \mathrm{d} x,$$
			with $J_{k_f-1}$ being  the Bessel function of the  first kind of  order $k_f-1$.
	\end{lemma}
	\begin{proof}
	See  Iwaniec-Kowalski \cite{IK}. 
	\end{proof}
Next we record the Voronoi  summation formula  for the Eisenstein series $E(\star,1/2)$ for $SL(2,\mathbb{Z})$. 
\begin{lemma}\label{voronoi for dn}
		Let $d(n)=\sum_{ab=|n|}1$ be the Fourier coefficients of $E(\star,1/2)$  and $g$ be a smooth, compactly supported function on $(0, \infty)$. Let $a$, $q \in \mathbb{Z}$ with $(a,q)=1$. Then we have
	$$\sum_{n=1}^{\infty} d(n) \, e\left(\frac{an}{q}\right)g(n) =I(g,q)+ \frac{1}{q}\sum_{\pm} \sum_{n=1}^{\infty} d(n) \, e\left(\mp\frac{\bar{a} n}{q}\right)\, h_{\pm}(n),$$
\end{lemma}
where $$I(g,q)=\frac{1}{q}\int_{0}^{\infty}(\log x+2\gamma-2\log q) g(x)\mathrm{d}x,$$
$$h_{+}(y)=\int_{0}^{\infty}-2\pi g(x) Y_0\left(\frac{4\pi \sqrt{xy}}{q}\right)\mathrm{d}x  \  \ \  \mathrm{and} \ \ \  h_{-}(y)=\int_{0}^{\infty}4 g(x)K_0\left(\frac{4\pi \sqrt{xy}}{q}\right)\mathrm{d}x.$$
	In the following lemma, we record some properties of $J_{k_f-1}$. 
	\begin{lemma} \label{bessel properties}
		Let $J_{k_f-1}(2\pi x)$ be the Bessel function of the first kind of integer order $k_f $.  Then, for fixed $k_f$, as $x \rightarrow \infty$, we have 
			$$J_{k_f-1}(2\pi x)= e(x)Z^{+}(x)+e(-x){Z^{-}}(x),$$
			where $Z^{-}(x)=\overline{Z^{+}(x)}$ and  $Z^{+}$ is a smooth function satisfying 
			$$ x^j {Z^+}^{(j)}(x) \ll_{j, \, k_f } \frac{1}{\sqrt{x}},$$
			for $j \geq 0$.
	\end{lemma}

	\subsection{Automorphic  forms on $GL(3)$}
	In this subsection, we will recall some background on the  Maass forms for $GL(3)$. This subsection, except for the notations, is taken from \cite{Li}. Let $\pi$ be a Hecke-Maass cusp form of type $(\nu_{1}, \nu_{2})$ for $SL(3, \mathbb{Z})$. Let $A(n,k)$ denote the normalized Fourier coefficients of $\pi$. Let 
	$${\alpha}_{1} = - \nu_{1} - 2 \nu_{2}+1, \, {\alpha}_{2} = - \nu_{1}+ \nu_{2},  \, {\alpha}_{3} = 2 \nu_{1}+ \nu_{2}-1$$ 
	be the Langlands parameters for $\pi$ (see Goldfeld \cite{DG} for more details).
	Let $g$ be a compactly supported smooth function on  $ (0, \infty )$ and $\tilde{g}(s) = \int_{0}^{\infty} g(x) x^{s-1} \mathrm{d}x$ be its Mellin transform. For $\ell= 0$ and $1$, we define
	\begin{equation*}
		\gamma_{\ell}(s) :=  \frac{\pi^{-3s-\frac{3}{2}}}{2} \, \prod_{i=1}^{3} \frac{\Gamma\left(\frac{1+s+{\alpha}_{i}+ \ell}{2}\right)}{\Gamma\left(\frac{-s-{\alpha}_{i}+ \ell}{2}\right)}.
	\end{equation*}
	Set $\gamma_{\pm}(s) = \gamma_{0}(s) \mp \gamma_{1}(s)$ and let 
	\begin{align}\label{gl3 integral transform}
		G_{\pm}(y) = \frac{1}{2 \pi i} \int_{(\sigma)} y^{-s} \, \gamma_{\pm}(s) \, \tilde{g}(-s) \, \mathrm{d}s,
	\end{align}
	where $\sigma > -1 + \max \{-\Re({\alpha}_{1}), -\Re({ \alpha}_{2}), -\Re({\alpha}_{3})\}$. With the aid of the above terminology, we now state the $GL(3)$ Voronoi summation formula in the following lemma:
	\begin{lemma} \label{gl3voronoi}
		Let $g(x)$ and  $A(n,k)$ be as above. Let $a,\bar{a}, q \in \mathbb{Z}$ with $q \neq 0, (a,q)=1,$ and  $a\bar{a} \equiv 1(\mathrm{mod} \ q)$. Then we have
		\begin{align*} \label{GL3-Voro}
			\sum_{n=1}^{\infty} A(n,k) e\left(\frac{an}{q}\right) g(n) 
			=q  \sum_{\pm} \sum_{n_{1}|qk} \sum_{n_{2}=1}^{\infty}  \frac{A(n_1,n_2)}{n_{1} n_{2}} S\left(k \bar{a}, \pm n_{2}; qk/n_{1}\right) G_{\pm} \left(\frac{n_{1}^2 n_{2}}{q^3 k}\right),
		\end{align*} 
		where  $S(a,b;q)$ is the  Kloosterman sum which is defined as follows:
		$$S(a,b;q) = \sideset{}{^\star}{\sum}_{x \,\mathrm{mod} \, q} e\left(\frac{ax+b\bar{x}}{q}\right).$$
	\end{lemma}
	\begin{proof}
		See \cite{Li} for the proof. 
	\end{proof}
	The following lemma, which gives the Ramanujan conjecture on average, is also well-known.  
	\begin{lemma} \label{ramanubound}
		We have 
		$$\mathop{\sum \sum}_{n_{1}^{2} n_{2} \leq x} \vert A(n_{1},n_{2})\vert ^{2} \ll \, x,$$
		where the implied constant depends on the form $\pi$.
	\end{lemma}
	\begin{proof}
		For the proof, we refer to Goldfeld's book \cite{DG}.
	\end{proof}

	\subsection{ Delta method} \label{circlemethod}
	Let $\delta: \mathbb{Z} \to \{0,1\}$ be defined by
	\[
	\delta(n)= \begin{cases}
		1 \quad \text{if} \,\  n=0; \\
		0 \quad  \,  $\textrm{otherwise}$.
	\end{cases}
	\]
	The above delta symbol can be used to separate the oscillations involved in a sum. Further, we seek a Fourier expansion of $\delta(n)$. We mention here an expansion for $\delta(n)$ which is due to Duke, Friedlander and Iwaniec (see \cite{IK}). Let $L\geq 1$ be a large number. For $n \in [-2L,2L]$, we have
	
	\begin{align*} 
		\delta(n)= \frac{1}{Q} \sum_{1 \leq q \leq Q} \frac{1}{q} \, \sideset{}{^\star}{\sum}_{a \, \mathrm{mod} \, q} \, e \left(\frac{na}{q}\right) \int_{\mathbb{R}} g(q,x) \,  e\left(\frac{nx}{qQ}\right) \, \mathrm{d}x,
	\end{align*}
	where  $Q=2L^{1/2}$. The $\star$ on the $q$-sum indicates that the sum over $a$ is restricted by the condition $(a,q)=1$. The function $g$ is the only part in the above formula which is not explicitly given. Nevertheless, we only need the following  properties of $g$ in our analysis:
	\begin{align} \label{g properties}
	&1. \ g(q,x)=1+h(q,x), \quad \text{with} \ \ \  h(q,x)=O \left(\frac{Q}{q} \left(\frac{q}{Q}+|x|\right)^{B}\right), \notag \\
	&2. \  x^j \frac{\partial ^j}{\partial x^j}g(q,x) \ll \log Q \min \left\lbrace \frac{Q}{q}, \frac{1}{|x|}\right\rbrace, \notag  \\
	&3. \  g(q,x) \ll |x|^{-B},  \notag \\
	&4. \ 	\int_{\mathbb{R}}|g(q,x)| \mathrm{d}x \ll Q^{\epsilon},
\end{align}
for any $B>1$ and $j \geq 1$. Using the above properties of $g(q,x)$ we observe that the effective range of the   above integration over $x$   is $[-L^{\epsilon},L^{\epsilon}]$. We record the above observations in the following lemma.
\begin{lemma}\label{deltasymbol}
		Let $\delta$ be as above and $g$ be a function satisfying \eqref{g properties}. Let $L\geq 1$ be a large parameter. Then, for $n \in [-2L,2L]$, we have
		\begin{equation*} 
			\delta(n)= \frac{1}{Q} \sum_{1 \leq q \leq Q} \frac{1}{q} \, \sideset{}{^\star}{\sum}_{a \, \mathrm{mod} \, q} \, e \left(\frac{na}{q}\right) \int_{\mathbb{R}}W_1(x) g(q,x) \,  e\left(\frac{nx}{qQ}\right) \, \mathrm{d}x+O(L^{-2020}),
		\end{equation*}
		where $Q=2L^{1/2}$ and  $W_1(x)$ is a smooth bump function supported in $[-2L^{\epsilon},2L^{\epsilon}]$, with $W_1(x)=1$ for $x \in [-L^{\epsilon},L^{\epsilon}] $ and $W^{(j)}\ll_j 1$.
	\end{lemma}
	\begin{proof}
		For the proof, we refer to Chapter 20 of the book \cite{IK} by Iwaniec and Kowalski and Lemma 15 of the article \cite{BingH} by B. Huang.
	\end{proof}
	\subsection{Stationary phase method} In this subsection, we will recall some facts about the exponential integrals of the following form:
	$$I= \int_{a}^bg(x)e(f(x))\mathrm{d}x,$$
	where $f$ and $g$ are smooth real valued functions on $[a,b]$. 
	\begin{lemma}\label{derivative bound}
		Let $I$, $f$ and $g$ be as above.  Then, for $r \geq 1$, we have 
		\begin{align}
			I \ll \frac{\text{Var} \ g}{\min|f^{(r)}(x)|^{1/r}}, \notag
		\end{align} 
		where $\text{Var} $ is the total variation of $g$ on $[a,b]$. 
		Moreover, let $f^{\prime}(x)\geq B$ and $f^{(j)}(x) \ll B^{1+\epsilon}$ for $j \geq 2$ together with $\textrm{Supp}(g) \subset (a,b)$  and $g^{(j)}(x) \ll_{a,b,j} 1$. Then we have
		\begin{align}
			I \ll_{a,b,j,\epsilon}B^{-j+\epsilon}. \notag
		\end{align}
	\end{lemma}
	\begin{proof}
		Proof of the first part of the lemma is standard. For the second part, we use integration by parts. 
	\end{proof}
	The following lemma gives an asymptotic expression for $I$ when the stationary point exists.
	\begin{lemma} \label{stationaryphase}
		Let $0 < \delta < 1/10, X, Y, U, Q>0, Z:= Q+X+Y+b-a+1$, and assume that
		$$ Y \geq Z^{3 \delta}, \, b-a \geq U \geq \frac{Q Z^{\frac{\delta}{2}}}{\sqrt{Y}}.$$
		Assume that $g$  satisfies  
		$$g^{(j)}(t) \ll_{j} \frac{X}{U^{j}} \, \, \, \text{for} \,\,  j=0,1,2,\ldots.$$  
		Suppose that there exists unique $t_{0} \in [a,b]$ such that $h^{\prime}(t_{0})=0$, and the function $f$ satisfies
		$$f^{\prime \prime}(t) \gg \frac{Y}{Q^2}, \, \, f^{(j)}(t) \ll_{j} \frac{Y}{Q^{j}} \, \, \, \, \text{for} \, \, j=1,2,3,\ldots.$$
		Then we have
		$$I = \frac{e^{i f(t_{0})}}{\sqrt{f^{\prime \prime}(t_{0})}} \, \sum_{n=0}^{3 \delta^{-1}A} p_{n}(t_{0}) + O_{A,\delta}\left( Z^{-A}\right), \, p_{n}(t_{0}) = \frac{\sqrt{2 \pi} e^{\pi i/4}}{n!} \left(\frac{i}{2 f^{\prime \prime}(t_{0})}\right)^{n} G^{(2n)}(t_{0}),$$
		where 
		$$ G(t)=g(t) e^{i H(t)}, \text{and} \, H(t)= f(t)-f(t_{0})-\frac{1}{2} f^{\prime \prime}(t_{0})(t-t_{0})^2.$$
		Furthermore, each  $p_{n}$ is a rational function in $h^{\prime}, h^{\prime \prime}, \ldots,$ satisfying the derivative bound
		$$\frac{d^{j}}{dt_{0}^{j}} p_{n}(t_{0}) \ll_{j,n} X \left(\frac{1}{U^{j}}+ \frac{1}{Q^{j}}\right) \left( \left(\frac{U^2 Y}{Q^2}\right)^{-n} + Y^{-\frac{n}{3}}\right).$$
	\end{lemma}

	\section{ The set-up and sketch of the proof}
	In this section, we will give a set-up to prove Theorem \ref{r>2 thm}. We then give a rough outline for the proof.
\subsection{Approximate functional equation}
	  Let $\pi$, $f$  and $\chi$ be as defined in Theorem \ref{r>2 thm}.  As a first step, we  express $L(1/2, \,  \pi \times f \times \chi)$ in terms of an exponential sum.  In fact, we have the following lemma. 
	\begin{lemma}\label{AF}
		 Let $\pi$, $f$  and $\chi$ be as defined in Theorem \ref{r>2 thm}. Let $R=p^r$ be the conductor of $\chi$. Then, for any $\epsilon >0$, as $R \rightarrow \infty$,  we have
		\begin{align} \label{AFE}
			L \left( \frac{1}{2}, \, \pi \times f \times \chi \right) \ll_{}\sup_{k \ll R^{{3}/{2}+\epsilon}} \, \sup_{ N \leq \frac{R^{3+\epsilon}}{k^2}} \frac{|S_{k}(N)|}{k\sqrt{N}} + R^{-2020},
		\end{align} 
		where
		\begin{align}\label{S_k(N)}
			S_{k}(N): = \mathop{\sum}_{n =1}^{\infty} \, A(n,k) \, \lambda_{f}(n)\chi(n) \, W\left(\frac{n}{N}\right),
		\end{align}
	and  $W$ is a smooth function supported in $[1,\, 2]$ and satisfying $W^{(j)} \ll R^{j\epsilon}$.
	\end{lemma}
	\begin{proof}
		Proof follows by an application of the functional equation of $L \left( \frac{1}{2}, \pi \times f \times \chi \right)$. We refer to  Theorem 5.3 and  Proposition 5.4 of  \cite{IK} for more details. 
			\end{proof}

	Thus, to establish  Theorem \ref{r>2 thm},  we need to get some cancellations in $S_{k}(N)$  in \eqref{S_k(N)}. 
	\subsection{ Application of  delta symbol} \label{dfi and congruence equation}
	There are three oscillatory factors in the sum $S_{k}(N)$ in \ref{S_k(N)}. Our next task is to separate these oscillations. We accomplish it  using  delta method. 
To this end, we rewrite $S_k(N)$ as 
\begin{align*}
		S_k(N) =  \mathop{\sum \sum}_{\substack{n, \,  m=1 \\ n = m } }^{\infty} A(n,k) \lambda_f(m) \chi(m) W \left(\frac{n}{N}\right) U \left(\frac{m}{N}\right),
\end{align*}
 where $U$ is a smooth  bump function supported in $[1/2,5/2]$, with  $ U(x)=1 \, \text{for} \ x \in [1,2]$ and  $U^{(j)} \ll R^{j\epsilon}$.
Now we detect $n=m$ as fellows:
$$n=m \iff n \equiv m\, \mathrm{mod} \ p^{\ell} \ \ \ \mathrm{and} \ \ \ \frac{n-m}{p^{\ell}}=0,$$
where $\ell$ is a positive integer such that $1\leq \ell <r$ (to be chosen optimally later). This is a crucial step to achieve our goal.  Thus, $S_k(N)$  can be rewritten as  
	
	\begin{equation} \label{congtrick}
		S_k(N) =  \mathop{\sum \sum}_{\substack{n, \, m=1 \\ n \equiv m \, {\mathrm{mod}}\, p^\ell} }^{\infty} A(n,k) \lambda_f(m) \chi(m)\delta \left(\frac{n-m}{p^{\ell}}\right) W \left(\frac{n}{N}\right) U \left(\frac{m}{N}\right),
	\end{equation}
where $\delta$ is the delta symbol defined in Subsection \ref{circlemethod}.   Now,  detecting  the congruence equation $n \, \equiv \, m \, {\mathrm{mod}}\, p^\ell$  using the additive characters, i.e. 
$$\delta(n \, \equiv \, m \, {\mathrm{mod}}\, p^\ell)= \frac{1}{p^{\ell}}\sum_{b \, {\mathrm{mod}}\, p^{\ell}}e\left(\frac{b(n-m)}{p^{\ell}}\right),$$
 and  applying Lemma \ref{deltasymbol} with $Q=R^{\epsilon}\sqrt{{N}/{p^{\ell}}} $, we arrive at
	\begin{align}\label{aftercircle}
		S_k(N)=& \frac{1}{Q p^\ell} \int_{\mathbb{R}}W_1(x) \sum_{1\leq q \leq Q} \frac{g(q,x)}{q} \sideset{}{^\star} \sum_{a \, {\mathrm{mod}}\, q} \sum_{b \, {\mathrm{mod}}\, p^{\ell}} \notag \\
		& \times \sum_{n=1}^{\infty} A(n,k) e\left(\frac{(a+bq)n}{p^\ell q}\right) e\left(\frac{nx}{p^\ell q Q}\right) W\left(\frac{n}{N}\right) \notag \\
		& \times \sum_{m=1}^{\infty} \lambda_f(m) \chi(m) e \left(\frac{-(a+bq)m}{p^\ell q}\right) e\left(\frac{-mx}{p^\ell qQ} \right)U\left(\frac{m}{N}\right) \mathrm{d}x+O(R^{-2020}). 
	\end{align}
	
	\subsection{Sketch of the proof} In this subsection, we will give a rough sketch of the proof. For simplicity, let's  consider the generic cases, i.e., $N=R^3$, $n \asymp N$, $m \asymp N$, $k=1$, $|x| \asymp 1$, $q \asymp  Q$ and $(q,p)=1$ in \eqref{aftercircle}.  Thus we essentially have the following expression for $S_k(N)$:
	\begin{align}\label{sketchy SN}
		\frac{1}{Q^2 p^\ell}\sum_{b \, {\rm mod}\, p^{\ell}}\sum_{q \asymp Q} \sideset{}{^\star}\sum_{a \, {\rm mod}\, q} \sum_{n \sim N} A(n,1) e\left(\frac{(a+bq)n}{p^\ell q}\right) \sum_{m \sim N} \lambda_f(m) \chi(m) e \left(\frac{-(a+bq)m}{p^\ell q}\right).
	\end{align}
Notice that we have   ignored the integral over $x$, as it has no oscillations in the generic case. On estimating the above expression trivially,  we get $S_k(N) \ll N^2$. Our aim is show $S_k(N) \ll \sqrt{N}R^{3/2-3/20}$. In other words, we need to save $N^2/(\sqrt{N}R^{3/2-3/20})=NR^{3/20}$ over the trivial bound $N^2$ in \eqref{sketchy SN}.  \\
 Our next step is to  dualize the sum over $n$ and $m$ using summation formulae. We accomplish it in Section \ref{summation formula}. In fact,  on applying the $GL(3)$ Voronoi formula to the sum over $n$ in \eqref{sketchy SN}, we, roughly, arrive at the following expression: 
 $$S_1:=\sum_{n \sim N} A(n,1) e\left(\frac{(a+bq)n}{p^\ell q}\right) \approx \frac{N^{2/3}}{Q p^{\ell}} \sum_{n_{2} \ll \sqrt{Np^{3\ell}}} \frac{A(1,n_{2})}{n_{2}^{1/3}} \, S\left(\overline{(a+bq)}, n_{2};qp^{\ell}\right).$$
 See Subsection \ref{gl3 formula} for more details. An application of  the $GL(2)$ Voronoi formula to the $m$-sum gives us 
 $$S_2:=\sum_{m \sim N} \lambda_f(m) \chi(m) e \left(\frac{-(a+bq)m}{p^\ell q}\right)\approx \frac{N^{3/4}}{ p^r  \sqrt{q}} 	\sideset{}{^\star}\sum_{\beta \, {\rm mod} \, p^r} \bar{\chi}(-\beta) \sum_{ m \ll p^{2r-\ell}} \frac{\lambda_{f}(m)}{m^{1/4}}e\left(\frac{\overline{c} m}{p^{r} q}\right),$$
 where $c=p^{r-\ell}(a+bq)+q \beta$. See Subsection \ref{gl2 formula} for full details.  Thus, we arrive at the following expression of $S_k(N)$:
 $$S_k(N) \rightsquigarrow \frac{N^{17/12}}{Q^{7/2}p^{r+2\ell}}\sum_{q \asymp Q}\sum_{n_{2} \ll \sqrt{Np^{3\ell}}} \frac{A(1,n_{2})}{n_{2}^{1/3}} \sum_{ m \ll p^{2r-\ell}} \frac{\lambda_{f}(m)}{m^{1/4}}\mathcal{C}(...),$$
 where $$\mathcal{C}(...)= \sideset{}{^\star}\sum_{a \, {\mathrm{mod}}\, q} \sum_{b \, {\mathrm{mod}}\, p^{\ell}}  	\sideset{}{^\star}\sum_{\substack{\beta \, {\rm mod} \, p^{r} }} \bar{\chi}\left(-\beta\right) S\left(\overline{(a+bq)}, n_{2};qp^{\ell}\right)e\left(\frac{\overline{c} m}{p^{r} q}\right),$$
 in which we seek square root cancellations. We analyze it in Section \ref{sum over a and b} and we get the following expression:
 $$\mathcal{C}(...) \approx p^{(r+\ell)/2}qe\left(-\frac{ n_2\overline{m} }{q }\right) \mathcal{C}_1(...),$$
  where $\mathcal{C}_1(...)$ is a character sum modulo $p^{\ell}$ in which we still need to get square root cancellations which we get  in Section \ref{rest cancellations for c}.  In the next step we apply the  Cauchy inequality followed by the Poisson to the sum over $n_2$ (See Section \ref{cauchy and poisson} for details).  The Cauchy inequality transforms $S_k(N)$ into 
 $$S_k(N) \ll \frac{N^{17/12}(Np^{3\ell})^{1/12}}{Q^{5/2}p^{r/2+3\ell/2}} \left(\sum_{n_{2} \ll \sqrt{Np^{3\ell}}} \Big| \sum_{q \asymp Q} \sum_{ m \ll p^{2r-\ell}} \frac{\lambda_{f}(m)}{m^{1/4}}e\left(-\frac{ n_2\overline{m} }{q }\right) \mathcal{C}_1(...)\Big|^2\right)^{1/2}.$$
 Next  we apply the Poisson summation formula to the sum over $n_2$ (see Subsection \ref{pois}). We observe that the  ``arithmetic  conductor" is of size $p^{\ell} Q^2$. Thus  we see that the sum over $n_2$ transfers into 
 $$\frac{\sqrt{Np^{3\ell}}}{p^{r+\ell/2} Q^2} \mathop{\sum \sum}_{q, q^{ \prime} \sim Q} \mathop{\sum \sum}_{m,m^\prime \sim p^{2r-\ell}} \sum_{n_{2} \ll \frac{p^\ell Q^2}{\sqrt{Np^{3\ell}}} } \left|\mathfrak{C}(...)\right|, $$
 where 
 \begin{align*}
 	\mathfrak{C}(...)&= \sum_{\nu_1 \, {\rm mod} \; p^\ell  }\mathcal{C}_1 \left(...\right) \overline{\mathcal{C}_1 \left(...\right)}e \left(\frac{\nu _1 n_{2}}{p^{\ell} }\right)\sum_{\nu_2 \, {\rm mod} \; qq^\prime  } e \left(\frac{(\overline{m^\prime}q-\overline{m}q^\prime +\nu _2) n_{2}}{qq^\prime}\right),
 \end{align*}
  which we analyze  in  Section \ref{rest cancellations for c}. For $n_2=0$,  we get $q=q^\prime$, $m=m^\prime$ (essentially) and 
$$\mathfrak{C}_0(...) \ll p^{2\ell}Q^2.$$
Hence, for $n_2=0$, we get (see Section \ref{omega zero and nonzero } for more  details) 
\begin{align}
	S_{k,0}(N) \ll \frac{N^{17/12}(Np^{3\ell})^{1/12}}{Q^{5/2}p^{r/2+3\ell/2}}\frac{(Np^{3\ell})^{1/4}}{p^{r/2+\ell/4}Q} \left(Qp^{2r-\ell}p^{2\ell}Q^2 \right)^{1/2}\ll R^{\epsilon}N^{1/2}p^{3r/4+3\ell/4}.
\end{align}
For, $n_2 \neq 0$, we analyze $\mathfrak{C}(...)$ differently  in Section \ref{rest cancellations for c} and we get (on average)
$$\mathfrak{C}_{\neq 0}(...) \ll p^{3\ell/2}.$$ 
We note that, we are  saving $q$  extra due to the additive character $e\left(-{ n_2\overline{m} }/{q }\right)$ which we gives  us a congruence condition modulo $qq^\prime$. Thus we get (see Section \ref{omega zero and nonzero } for more  details) 
\begin{align}
	S_{k,\neq 0}(N) \ll \frac{N^{17/12}(Np^{3\ell})^{1/12}}{Q^{5/2}p^{r/2+3\ell/2}}\frac{(Np^{3\ell})^{1/4}}{p^{r/2+\ell/4}Q} \frac{(p^\ell Q^2)^{1/2}}{(Np^{3\ell})^{1/4}}\left(Q^2p^{4r-2\ell}p^{3\ell/2} \right)^{1/2}\ll R^{\epsilon}{p^{r-\ell/2}kN^{3/4}}.
\end{align}

	Upon choosing $\ell$  optimally, we get 
	 $$S_k(N)\ll \sqrt{N}p^{3r/4+3[4r/5]/4+\epsilon},$$
	and consequently, 
	$$L(1/2,\pi \times f \times \chi) \ll p^{3r/4+3[4r/5]/4+\epsilon}\ll(p^r)^{3/2- 3/20+\epsilon} .$$
	Hence  we get Theorem \ref{r>2 thm}. 
	
	\section{Voronoi summation formulae} \label{summation formula}
	
	In this section, we apply summation formulae to the $n$-sum  and $m$-sum in \eqref{aftercircle}. 
\subsection{$GL(3)$ Voronoi formula} \label{gl3 formula} In this subsection, we will  analyze the sum over $n$
\begin{align}\label{n sum }
	S_1:=\sum_{n=1}^{\infty} A(n,k) e\left(\frac{(a+bq)n}{p^{\ell} q}\right) e\left(\frac{nx}{p^{\ell}qQ}\right) W\left(\frac{n}{N}\right)
\end{align}
in \eqref{aftercircle} using  the $GL(3)$ Voronoi summation formula.  Let $q=p^{\ell ^\prime}q^\prime $ with $(p,q^\prime)=1$ and $\ell^\prime \geq 0$. Thus it follows that 
$$(a+bq, \,p^\ell q)=(a+bq,\, p^{\ell+\ell^\prime}),$$
as $(a,\, q)=1$. 
Moreover, if $\ell^\prime >0$, then $(a+bq,\, p^{\ell+\ell^\prime})=1$. In the other case, i.e., $\ell^\prime =0$ or $(q,\, p)=1$,  let $(a+bq, \,p^\ell q)=p^{\ell_1}$, with $0 \leq \ell_1 \leq \ell$. On applying  Lemma \ref{gl3voronoi} to the $n$-sum in \eqref{aftercircle} with the modulus $p^{\ell -\ell_1}q$ and  $g(n)=e\left({nx}/{p^\ell q Q}\right) W\left({n}/{N}\right)$, we arrive at 
\begin{align}\label{dual gl3}
	p^{\ell-\ell_1} q \sum_{\pm} \sum_{n_{1} \vert \frac{ p^{\ell}qk}{p^{\ell_1}}} \sum_{n_{2}=1}^{\infty} \frac{A(n_{1},n_{2})}{n_{1}n_{2}} S(k(\overline{(a+bq)/p^{\ell_1}}), \pm n_{2};qp^{\ell-\ell_1}k/n_{1}) G_{\pm} \left(\frac{n_{1}^{2} n_{2}}{(q p^{\ell-\ell_1})^{3}k}\right).
\end{align}
Next we will analyze the integral transform $G_{\pm}(y)$. We have the following lemma. 
\begin{lemma}
	Let $G_{\pm}(y)$ be the integral transform as defined in \eqref{gl3 integral transform}. Let $y=\frac{n_1^2n_2}{(qp^{\ell-\ell_1})^3k}$. Then 
	$G_{\pm}(y)$ is negligibly small unless  $n_1^2n_2 \ll  N^{1/2}p^{3\ell/2-3\ell_1}kR^{\epsilon} =: N_{0}$. In this range, we have 
	\begin{align}
		G_{\pm} \left(\frac{n_1^2n_2}{(qp^{\ell-\ell_1})^3 k}\right) = \left(\frac{n_1^2n_2N}{(qp^{\ell-\ell_1})^3 k}\right)^{1/2}  	I_{1}(n_1^2n_2,q,x),
	\end{align}
where $I_{1}(n_1^2n_2,q,x)$ is a integral transform defined in \eqref{integral of gl3}.
\end{lemma}
\begin{proof}
	Let's recall from \eqref{gl3 integral transform} that 
	\begin{align}\label{Gy}
	G_{\pm}(y) &= \frac{1}{2 \pi i} \int_{(\sigma)} y^{-s} \, \gamma_{\pm}(s) \, \int_{0}^{\infty} e\left(\frac{xz}{p^{\ell}qQ}\right) W\left(\frac{z}{N}\right)z^{-s-1} \, \mathrm{d}z\, \mathrm{d}s  \notag\\
	&= \frac{1}{2 \pi } \int_{-\infty}^{\infty} (Ny)^{-\sigma-i\tau} \, \gamma_{\pm}(\sigma+i\tau) \, \int_{0}^{\infty} e\left(\frac{Nxz}{p^{\ell}qQ}\right)W(z)z^{-\sigma-1-i\tau} \, \mathrm{d}z\, \mathrm{d}\tau 
\end{align}
	On applying integration by parts, we infer that the $z$-integral is negligibly small unless \begin{align}\label{size of tau}
		|\tau| \asymp N|x|/(p^{\ell} qQ).
	\end{align}
	Using the  Stirling formula, for $\sigma \geq -1/2$, we have  
	$$\gamma_{\pm}(\sigma+i\tau) \ll_{\pi,\sigma} (1+|\tau|)^{3(\sigma+1/2)}.$$
	Thus, on plugging this bound into \eqref{Gy}, we get 
\begin{align}
		G_{\pm}(y) \ll \left(\frac{N}{p^{\ell}qQ}\right)^{5/2}\left(\frac{N^3}{Nyp^{3\ell}q^3Q^3}\right)^{\sigma} \ll \left(\frac{Q}{q}\right)^{5/2}\left(\frac{Q^3}{Nyq^3}\right)^{\sigma}. 
\end{align}
Thus, on moving the contour to  $\sigma $ sufficiently large (towards $\infty$) and  taking $y=\frac{n_1^2n_2}{(qp^{\ell-\ell_1})^3k}$, we see that $G_{\pm}(y) $ is negligibly small if 
\begin{align}\label{N0}
	n_1^2n_2 \gg R^{\epsilon}\frac{(p^{\ell-\ell_1}Q)^{3}k}{N}=N^{1/2}p^{3\ell/2-3\ell_1}kR^{\epsilon} =: N_{0}.
\end{align}
In the complimentary range, we move the contour to $\sigma=-1/2$, to get 
\begin{align}
		G_{\pm} \left(\frac{n_1^2n_2}{(qp^{\ell-\ell_1})^3 k}\right) = \left(\frac{n_1^2n_2N}{(qp^{\ell-\ell_1})^3 k}\right)^{1/2}  	I_{1}(n_1^2n_2,q,x) ,
\end{align}
where 
\begin{align}\label{integral of gl3}
	I_{1}(n_1^2n_2,q,x):=	\frac{1}{2 \pi } \int_{-\infty}^{\infty}  \left(\frac{n_1^2n_2N}{(qp^{\ell-\ell_1})^3 k}\right)^{-i\tau} \, \gamma_{\pm}(-1/2+i\tau)\widetilde{W}(\tau, x )  \mathrm{d}\tau 
\end{align}
and 
\begin{align}
	\widetilde{W}(\tau, x )= \int_{0}^{\infty} e\left(\frac{Nxz}{p^{\ell}qQ}\right)W(z)z^{-1/2-i\tau} \, \mathrm{d}z. 
\end{align}
Hence, we have the lemma. 
\end{proof}
Thus, on applying the above lemma  to \eqref{dual gl3}, we arrive at  
\begin{align*}
	S_1= \frac{\sqrt{N}p^{\ell_1/2}}{(qp^{\ell})^{1/2} k^{1/2}}  \sum_{\pm} \sum_{n_{1} \vert \frac{ p^{\ell}qk}{p^{\ell_1}}} \sum_{n_{2}=1}^{\infty} \frac{A(n_{1},n_{2})}{n_{2}^{1/2}} S(k(\overline{(a+bq)/p^{\ell_1}}), \pm n_{2};qp^{\ell-\ell_1}k/n_{1})I_1(...).
\end{align*}
We conclude  this subsection by recording the above analysis in the following lemma. 
\begin{lemma}\label{gl3 }
	Let $S_1$ be as in \eqref{n sum }. Let $(a+bq,p^{\ell}q)=p^{\ell_1}$. Then,  we have 
\begin{align*}
	S_1= \frac{\sqrt{N}p^{\ell_1/2}}{(qp^{\ell})^{1/2} k^{1/2}}  \sum_{\pm} \sum_{n_{1} \vert \frac{ p^{\ell}qk}{p^{\ell_1}}} \sum_{n_{2}=1}^{\infty} \frac{A(n_{1},n_{2})}{n_{2}^{1/2}} S(k(\overline{(a+bq)/p^{\ell_1}}), \pm n_{2};qp^{\ell-\ell_1}k/n_{1})I_1(...),
\end{align*}
	where $N_0$ and $ I_{1}(...)$ are as in \eqref{N0} and \eqref{integral of gl3} respectively. 
\end{lemma}

	\subsection{$GL(2)$ Voronoi summation formula} \label{gl2 formula}
In this step, we apply the $GL(2)$ Voronoi formula to the sum over $m$ in \eqref{aftercircle}.  In fact, we have the following lemma.
\begin{lemma}\label{gl2}
	Let $S_2$ denotes the sum over $m$ in \eqref{aftercircle}. Let $q=q^\prime p^{\ell^\prime}$. Then we have  
	\begin{align} \label{gasussum}
		S_2=\frac{2\pi i^{k_f}N^{3/4}}{ \tau(\bar{\chi})  \sqrt{qp^{r-\ell_2}}} 	\sideset{}{^\star}\sum_{\beta \, {\rm mod} \, p^r} \bar{\chi}(-\beta) \sum_{1\leq m \ll M_0} \frac{\lambda_{f}(m)}{m^{1/4}}e\left(\frac{\overline{(c/p^{\ell_2})} m}{p^{r-\ell_2} q}\right)  I_{2}(q,m,x),
	\end{align}
	where $c=p^{r-\ell}(a+bq)+q \beta$,  $(c,p^rq)=p^{\ell_2}$,  $M_0=R^{\epsilon}p^{2r-\ell-2\ell_2}$ and
	\begin{align} \label{integral of gl2 }
		I_2(q,m,x)= \int_{0}^{\infty } U(y)e\left(\frac{-Nxy}{p^{\ell}qQ}\right) e\left(\pm\frac{2\sqrt{Nmy}}{p^{r-\ell_2} q}\right) \mathrm{d}y.
	\end{align}
\end{lemma}
\begin{proof}
	Firstly, we expand $\chi(m)$ in terms of additive characters. In fact, we have 
	\begin{equation*}
		\chi(m) = \frac{1}{\tau(\bar{\chi})} \sum_{\beta \, \mathrm{mod} \, p^r} \bar{\chi}(\beta) e\left(\beta m/p^r\right),
	\end{equation*}
	where $\tau{(\bar{\chi})}$ is the Gauss sum associated to $\bar{\chi}$. Therefore the $m$-sum in \eqref{aftercircle} transforms into
	\begin{align}\label{m sum after expanding x}
		S_{2}:&=\sum_{m=1}^{\infty} \lambda_f(m) \chi(m) e \left(\frac{-(a+bq)m}{p^\ell q}\right) e\left(\frac{-mx}{p^\ell qQ} \right)U\left(\frac{m}{N}\right) \\
		&= \frac{1}{\tau{(\bar{\chi})}} \sum_{\beta \, \mathrm{mod} \, p^r} \bar{\chi}(-\beta) \sum_{m=1}^{\infty} \lambda_f(m)  e \left(\frac{-(a+bq)m}{p^\ell q} - \frac{\beta m }{p^r}\right) e\left(\frac{-mx}{p^\ell qQ} \right)U\left(\frac{m}{N}\right).
	\end{align}
Observe that, if $(\beta, \, p)>1$,  then $S_2=0$. Thus we can assume that $(\beta, \, p)=1$. Let's  consider  $c=p^{r-\ell}(a+bq)+q \beta$. Let 
 $(c, \, p^rq)=(c, \, p^{r+\ell^\prime})=p^{\ell_2}$. Note that,  if $\ell^\prime =0$, then $\ell_2=0$. On applying  Lemma \ref{gl2voronoi} to $S_2$ with the modulus $p^{r-\ell_2}q$ and  $g(m)=e\left(-mx/(p^\ell qQ)\right)$ $U(m/N)$, we arrive at
	
	\begin{align}\label{S_2 after gl2}
		S_{2}= &\frac{2\pi i^{k_f}}{\tau{(\bar{\chi})} p^{r-\ell_2} q} 	\sideset{}{^\star} \sum_{\beta \, {\rm mod} \, p^r} \bar{\chi}(-\beta) \sum_{m=1}^{\infty} \lambda_f(m) e\left(\frac{\overline{(c/p^{\ell_2})} m}{p^{r-\ell_2} q}\right)h(m) ,
	\end{align}
	where 
	\begin{align}
		h(m)= \int_{0}^{\infty} U\left(\frac{y}{N}\right) e\left(\frac{yx}{p^{\ell}qQ}\right) J_{k_f-1} \left(\frac{ 4 \pi \sqrt{my}}{p^{r-\ell_2} q}\right) \mathrm{d}y.
	\end{align}
	Upon changing the variable $y \mapsto Ny$ and extracting the oscillations of $J_{k_f-1}$ using Lemma \ref{bessel properties}, we observe  that $h(m)$ has essentially the following expression
	\begin{align}\label{h(m)}
		h(m)&=N\int_{0}^{\infty } U(y)Z^{\pm}\left(\frac{2\sqrt{Nmy}}{p^{r-\ell_2} q} \right)e\left(\frac{-Nxy}{p^{\ell}qQ}\right) e\left(\pm\frac{2\sqrt{Nmy}}{p^{r-\ell_2} q}\right) \mathrm{d}y \notag \\
		&=\frac{N^{3/4}p^{\frac{r-\ell_2}{2}}\sqrt{q}}{m^{1/4}} \int_{0}^{\infty } U(y)e\left(\frac{-Nxy}{p^{\ell}qQ}\right) e\left(\pm\frac{2\sqrt{Nmy}}{p^{r-\ell_2} q}\right) \mathrm{d}y \notag  \\
		&:=\frac{N^{3/4}p^{\frac{r-\ell_2}{2}}\sqrt{q}}{m^{1/4}}I_2(q,m,x). 
	\end{align}
	Note the abuse  of notation in here.  The weight function  $U$ appearing above is different from the one we started with.  But the new $U$  still  satisfies $U^{(j)}(x) \ll_{j, k_f} 1/x^j $  and support$(U) \subset [1/2,5/2]$.  On applying integration by parts, we observe that 
	\begin{equation*}
		I_{2}(q,m,x) \ll_{j} \left(1+\frac{N|x|}{p^{\ell}qQ}\right)^{j} \left(\frac{p^{r-\ell_2} q}{\sqrt{Nm}}\right)^j \ll \left(1+\frac{NR^{\epsilon}}{p^{\ell}qQ}\right)^{j} \left(\frac{p^{r-\ell_2} q}{\sqrt{Nm}}\right)^j  .
	\end{equation*}
	Thus the integral $I_2(...)$ is negligibly small if 
	$$m \gg R^{\epsilon}\max\left(\frac{(p^{r-\ell_2} q)^2 }{N},\ p^{2r-\ell-2\ell_2}\right)=R^{\epsilon}p^{2r-\ell-2\ell_2}:=M_{0}.$$
	Now  plugging the   expression  \eqref{h(m)} of $h(m)$ into \eqref{S_2 after gl2}, we get the lemma. 
	\end{proof}
	\begin{remark}\label{on gl2 eisenstein}
		In the above $m$-sum \eqref{m sum after expanding x}, if we had $d(m)$ instead of $\lambda_{f}(n)$, then on applying the Voronoi formula for $d(m)$, Lemma \ref{voronoi for dn}, we get a main term which would vanish, as it does not involve $\beta$ (see Lemma \ref{voronoi for dn}), and hence the sum over $\beta$ will vanish. The ramining part of the Voronoi summation formula for $d(m)$ is similar to the Voronoi formula for $\lambda_{f}(m)$.  Hence, following the similar arguments, we  also get subconvexity bounds for $L(1/2,\pi \times E\times \chi)$. 
	\end{remark}

\subsection{$S_k(N)$ after summation formulae} We conclude this section by combining Lemma \ref{gl3 } and Lemma \ref{gl2}. 
\begin{lemma} \label{S(N) after voronoi}
	Let $S_k(N)$ be as in \eqref{aftercircle}. Then we have 
		\begin{align}\label{SN before cauchy}
		S_k(N)=&\frac{2\pi i^{k_f}p^{\ell_1/2+\ell_2/2}N^{5/4}}{ \tau \left(\bar{\chi}\right)Q p^{(3\ell + {r})/{2}}k^{1/2}}  \mathop{\sum}_{\substack{1\leq q \leq Q }} \frac{1}{q^{2}} \sum_{\pm}\sum_{n_{1}|qp^{\ell-\ell_1}k} \notag   \\
		& \times  \sum_{n_{2} \ll N_0/n_1^2} \frac{A(n_{1},n_{2})}{n_{2}^{1/2}}  \sum_{ m \ll M_0} \frac{\lambda_{f}(m)}{m^{1/4}}\mathcal{C}(...) 	\mathcal{J}\left(n_1^2n_2,q,m\right),
	\end{align}
	where 
	\begin{align}\label{integral}
		\mathcal{J}\left(n_1^2 n_2,q,m\right):&=\int_{\mathbb{R}}W_1(x)g(q,x)I_{1}(n_{1}^2 n_{2},q,x)I_{2}\left(q,m,x\right)\mathrm{d}x, 
	\end{align}
and the character sum $	\mathcal{C} \left(...\right)$ is defined as 
	\begin{align}
 \sideset{}{^\star}\sum_{a \, {\mathrm{mod}}\, q} \sum_{b \, {\mathrm{mod}}\, p^{\ell}}  	\sideset{}{^\star}\sum_{\substack{\beta \, {\rm mod} \, p^{r} }} \bar{\chi}\left(-\beta\right) S\left(k(\overline{(a+bq)/p^{\ell_1}}),\pm n_{2};qp^{\ell-\ell_1}k/n_{1}\right)e\left(\frac{\overline{(c/p^{\ell_2})} m}{p^{r-\ell_2} q}\right).
\end{align}
	\end{lemma}
\begin{proof}
Proof follows by plugging   Lemma \ref{gl3 } and Lemma \ref{gl2} into \eqref{aftercircle}.
\end{proof}
		\section{The sum over $a$, $\beta$ and $b$} \label{sum over a and b}
	In this section, we will analyze  the character sum $\mathcal{C}(...)$ defined in Lemma \ref{S(N) after voronoi}. It  is given as 
	\begin{align}
		\sideset{}{^\star}\sum_{a \, {\mathrm{mod}}\, q} \sum_{b \, {\mathrm{mod}}\, p^{\ell}}    	\sideset{}{^\star}\sum_{\substack{\beta \, {\rm mod} \, p^{r} }} \bar{\chi}\left(-\beta\right) S\left(k(\overline{(a+bq)/p^{\ell_1}}),\pm n_{2};qp^{\ell-\ell_1}k/n_{1}\right)e\left(\frac{\overline{(c/p^{\ell_2})} m}{p^{r-\ell_2} q}\right).
	\end{align}
	Let $k=p^{\ell_3 } k^\prime$ and  $n_1=p^{\ell_4} n_1^\prime $  with $(k^\prime, \, p)=(n_1^\prime,\, p)=1$.  Thus $qp^{\ell-\ell_1}k/n_1=({q^\prime k^\prime}/{n_1^\prime})p^{\ell_5}$, where $\ell_5=\ell+\ell^\prime +\ell_3 -\ell_1-\ell_{4}$. On splitting the Kloosterman sum using reciprocity, we arrive at 
	\begin{align}\label{C before cauchy}
		\mathcal{C}(...)&=\sideset{}{^\star} \sum_{\alpha \, {\mathrm{mod}}\, p^{\ell_5} q^\prime k^\prime/n_1^\prime}e\left(\frac{\pm n_2\overline{\alpha}}{p^{\ell_5}qk^\prime /n_1^\prime}\right)  \sideset{}{^\star}\sum_{a \, {\mathrm{mod}}\, q}   e\left(\frac{\alpha k\overline{a }p^{\ell_1}\overline{p^{\ell_5}}}{ q^\prime k^\prime /n_1^\prime}\right) e\left(\frac{\overline{(p^{2r-2\ell_2-\ell+\ell^\prime}a)} m}{ q^\prime}\right) \notag \\
		&\times \sum_{b \, {\mathrm{mod}}\, p^{\ell}}   e\left(\frac{k\alpha \overline{((a+bq)/p^{\ell_1})} \overline{{(q^\prime k^\prime /n_1^\prime) }} \, }{p^{\ell_5}  }\right)	\sideset{}{^\star}\sum_{\substack{\beta \, {\rm mod} \, p^{r} }} \bar{\chi}\left(-\beta\right) e\left(\frac{\overline{(c/p^{\ell_2})} m\overline{q^\prime}}{p^{r-\ell_2+\ell^\prime} }\right).
	\end{align}
We have the following lemma. 
\begin{lemma}\label{characer for r even}
	Let $\mathcal{C}(...)$ be as in  \eqref{C before cauchy}.  Then, if $(q,p)=1$,  we have 
	\begin{align} \label{character sum  before cauchy }
		&\mathcal{C}(...)=p^{(r+\ell-\ell_1)/2}\chi(-q)\sum_{d \vert q} d \mu \left(\frac{q}{d}\right)  \sideset{}{^\star} \sum_{\substack{\alpha \; {\rm mod} \; p^{\ell_5}qk^\prime/n_1^\prime \\ h_1(\alpha,m)\equiv \; 0 \; {\rm mod} \; d}} e\left(\frac{\pm n_2\overline{\alpha} }{p^{\ell_5}q k^\prime/n_1^\prime}\right)  \notag  \\
		&\times \mathop{\sideset{}{^\star} \sum_{u \, {\rm mod} \, p^{(\ell-\ell_1)/2}} \ \sideset{}{^\star} \sum_{v\, {\rm mod} \, p^{r/2}}}_{\substack { h_2(v, u,m) \equiv \; 0 \; {\rm mod} \; p^{r/2}  \\ h_3(v,\, u, \, \alpha) \equiv \; 0 \; {\rm mod} \; p^{(\ell-\ell_1)/2}}} 
		\overline{\chi}\left(v-p^{r-\ell+\ell_1} u\right)e\left(\frac{\alpha \overline{u} \overline{q}n_1 }{p^{\ell-\ell_1}}  \right) e\left(\frac{\overline{v} m\overline{q}}{p^{r} }\right),
	\end{align}
where $h_1(\alpha,m)$, $h_2(v,u,m)$ and $h_3(v,u,\alpha)$ are defined as in \eqref{h_1}, \eqref{h_2} and \eqref{h_3} respectively.  
More specifically, the above sum vanishes  if $(n_1,\,p)>1$. In the other case, i.e., for $(q,p)>1$, we have 
	\begin{align} \label{char for non primes 1 }
	&\mathcal{C}(...)=p^{(r+\ell+\ell^\prime)/2}\chi(-q^\prime)\sum_{d \vert q^\prime} d \mu \left(\frac{q^\prime}{d}\right)  \sideset{}{^\star} \sum_{\substack{\alpha \; {\rm mod} \; p^{\ell_5}q^\prime k^\prime/n_1^\prime \\ h_1^\prime(\alpha,m)\equiv \; 0 \; {\rm mod} \; d}} e\left(\frac{\pm n_2\overline{\alpha} }{p^{\ell_5}q^\prime k^\prime/n_1^\prime}\right)  \notag  \\
	&\times \mathop{\sideset{}{^\star} \sum_{u \, {\rm mod} \, p^{(\ell+\ell^\prime)/2}} \ \sideset{}{^\star} \sum_{v\, {\rm mod} \, p^{r/2}}}_{\substack { h_2^\prime(v, u,m) \equiv \; 0 \; {\rm mod} \; p^{r/2}  \\ h_3^\prime(v,\, u, \, \alpha) \equiv \; 0 \; {\rm mod} \; p^{(\ell+\ell^\prime)/2}}} 
	\overline{\chi}\left(v-p^{r-\ell-\ell^\prime} u\right)e\left(\frac{\alpha \overline{u} \overline{q^\prime}n_1 }{p^{\ell+\ell^\prime}}  \right) e\left(\frac{\overline{v} m\overline{q^\prime}}{p^{r} }\right),
\end{align}

\end{lemma}
\begin{proof}

We will analyze $\mathcal{C}(...)$ in  two cases. 
	\subparagraph*{Case 1} $(q, \, p)=1$, i.e., $\ell^\prime=0$ and $q=q^\prime$. \\
		As $(q,\, p)=1$, we note that $(c,\, p^rq)=(p^{r-\ell}(a+bq)+\beta q, \, p^{r}q)=1$, and hence $\ell_2=0$.  Thus, by changing the variable $b \mapsto (a+bq)=u$ in \eqref{C before cauchy}, we get the following expression for $\mathcal{C}(...)$:
		\begin{align}\label{C for coprime q}
		\mathcal{C}(...)&=\sideset{}{^\star} \sum_{\alpha \, {\mathrm{mod}}\, p^{\ell_5} q k^\prime/n_1^\prime}e\left(\frac{\pm n_2\overline{\alpha}}{p^{\ell_5}qk^\prime /n_1^\prime}\right)  \sideset{}{^\star}\sum_{a \, {\mathrm{mod}}\, q}   e\left(\frac{\alpha k\overline{a }p^{\ell_1}\overline{p^{\ell_5}}}{ q k^\prime /n_1^\prime}\right)e\left(\frac{\overline{(p^{2r-\ell}a)} m}{ q}\right) \notag \\
		&\times \sideset{}{^\star} \sum_{u \, {\mathrm{mod}}\, p^{\ell-\ell_1}}   e\left(\frac{\alpha \overline{u} \overline{q}n_1 }{p^{\ell-\ell_1}}  \right)	\sideset{}{^\star}\sum_{\substack{\beta \, {\rm mod} \, p^{r} }} \bar{\chi}\left(-\beta\right)e\left(\frac{\overline{(p^{r-\ell+\ell_1}u+q \beta)} m \overline{q}}{p^{r} }\right).
	\end{align}
Let's consider the sum over $\beta$ in the above expression. It is given as 
	$$\mathcal{C}_{\beta}(...):=\sideset{}{^\star}\sum_{\substack{\beta \, {\rm mod} \, p^{r} }} \bar{\chi}\left(-\beta\right)e\left(\frac{\overline{(p^{r-\ell+\ell_1}u+q \beta)} m \overline{q}}{p^{r} }\right).$$
On  changing the variable  $p^{r-\ell+\ell_1}u+q \beta=v$, 
	we arrive at
	\begin{align*}
		\mathcal{C}_{\beta}(...)= \chi(-q)	\sideset{}{^\star}\sum_{\substack{v \, {\rm mod} \, p^{r} }} \overline{\chi}\left(v-p^{r-\ell+\ell_1}u\right) e\left(\frac{\overline{v} m\overline{q}}{p^{r} }\right).
	\end{align*}
	Let us first  assume that $r$ is even. Splitting the sum over $v$ into the  residue classes modulo $p^{r/2}$, i.e., writing $v=v_1+v_2p^{r/2}$, we get the following expression for $	\mathcal{C}_{\beta}(...)$:
	\begin{align*}
		 &\chi(-q)\sideset{}{^\star}\sum_{\substack{v_1 \, {\rm mod} \, p^{r/2} }}   \overline{\chi}\left(v_1-p^{r-\ell+\ell_1}u\right)	e\left(\frac{\overline{v_1} m\overline{q}}{p^{r} }\right) \\
		& \times \sum_{\substack{v_2 \, {\rm mod} \, p^{r/2} }}  \overline{\chi}\left(1+\overline{(v_1-p^{r-\ell+\ell_1}u)}v_2p^{r/2}\right) e\left(-\frac{\overline{v_1}^2v_2 m\overline{q}}{p^{r/2} }\right).
	\end{align*}
	We observe that ${\chi}(1+v_2p^{r/2})$ is an additive character modulo $p^{r/2}$ of order $p^{r/2}$. More precisely, we have 
	$${\chi}(1+v_2p^{r/2})=e\left(\frac{Av_2}{p^{r/2}}\right),$$
	for some constant $A$ such that $(A,p)=1$ which  depends only on the character $\chi$. Thus, on  evaluating the sum over $v_2$,  we get the following expression for $	\mathcal{C}_{\beta}(...)$:
	\begin{align}\label{Cb}
		\mathcal{C}_{\beta}(...)=p^{r/2} \chi(-q)\sideset{}{^\star} \sum_{\substack{v \, {\rm mod} \, p^{r/2} \\ h_2(v,u,m) \equiv \; 0 \; {\rm mod} \; p^{r/2} }} \overline{\chi}\left(v-p^{r-\ell+\ell_1}u\right) e\left(\frac{\overline{v} m\overline{q}}{p^{r} }\right),
	\end{align}
	where 
	\begin{align}\label{h_2}
		h_2(v,u,m)=Aq v^2 +mv-m p^{r-\ell+\ell_1}u.
	\end{align} 
 For  $r$  odd,  a similar analysis can be done. In fact, we get a similar sum as above. We refer to  Chapter 12 of \cite{IK}  for more details. Thus,  for simplicity, we will continue the proof  for $r$ even. 
	Next  we analyze  the  sum over $a$  in \eqref{C before cauchy}.  
	It is evalualed as 
	\begin{align}\label{sum over a}
		\sideset{}{^\star}\sum_{a \, {\mathrm{mod}}\, q}   e\left(\frac{\alpha k\overline{a }p^{\ell_1}\overline{p^{\ell_5}}}{ q k^\prime /n_1^\prime}\right)e\left(\frac{\overline{(p^{2r-\ell}a)} m}{ q}\right) 	= \sum_{d \vert(q, \, h_1(\alpha, m)) } d \mu \left(\frac{q}{d}\right),
	\end{align}
where 
	\begin{align}\label{h_1}
		h_1(\alpha,m)=n_1^\prime \alpha p^{2\ell_1+\ell_4}\overline{p^{2\ell}}  + m \overline{p^{2r}}.
	\end{align}
	Thus, on plugging \eqref{Cb} and \eqref{sum over a} into \eqref{C for coprime q}, we get the following expression for $\mathcal{C}(...)$:
		\begin{align}\label{character sum (l prime =0) after a sum }
		\mathcal{C}(...)&=p^{r/2}\chi(-q)\sum_{d \vert q} d \mu \left(\frac{q}{d}\right)  \sideset{}{^\star} \sum_{\substack{\alpha \; {\rm mod} \; p^{\ell_5}qk^\prime/n_1^\prime \\ h_1(\alpha,m)\equiv \; 0 \; {\rm mod} \; d}} e\left(\frac{\pm n_2\overline{\alpha} }{p^{\ell_5}q k^\prime/n_1^\prime}\right)   \notag\\
		&\times  \mathop{\sideset{}{^\star} \sum_{u \, {\rm mod} \, p^{(\ell-\ell_1)}} \ \sideset{}{^\star} \sum_{v\, {\rm mod} \, p^{r/2}}}_{\substack{ {h_2(v, u,m) \equiv \; 0 \; {\rm mod} \; p^{r/2} }}} \overline{\chi}\left(v-p^{r-\ell+\ell_1} u\right) e\left(\frac{\overline{v} m\overline{q}}{p^{r} }\right)e\left(\frac{\alpha \overline{u} \overline{q}n_1 }{p^{\ell-\ell_1}}  \right) .
	\end{align}
	Next we analyze the sum over $u$. It is given as
\begin{align}\label{sum over u and v}
	\mathcal{C}_u(...):&=\sideset{}{^\star} \sum_{\substack{u \, {\rm mod} \, p^{\ell-\ell_1} \\ h_2(v, u,m) \equiv \; 0 \; {\rm mod} \; p^{r/2} }} 
	\overline{\chi}\left(v-p^{r-\ell+\ell_1} u\right) e\left(\frac{\alpha \overline{u} \overline{q}n_1 }{p^{\ell-\ell_1}}  \right) 
\end{align}
Let's  assume (for simplicity) that $\ell-\ell_1$ is even. On reducing  $u$ modulo $(\ell-\ell_1)/2$, i.e., writing $u$ as
$$u=u_{1}+u_{2} p^{(\ell-\ell_1)/2} \iff \bar{u} = \bar{u_{1}} - \bar{u_{1}} ^{2} u_{2} p^{(\ell-\ell_1)/2},$$
we get the following expression for $ \mathcal{C}_{u}(...)$:
\begin{align*}
	&\sideset{}{^\star} \sum_{\substack{u_1 \, {\rm mod} \, p^{(\ell-\ell_1)/2} \\ h_2(v,\, u_1,m) \equiv \; 0 \; {\rm mod} \; p^{r/2} }} 
	\overline{\chi}\left(v-p^{r-\ell+\ell_1} u_1\right) e\left(\frac{\alpha \overline{u_1} \overline{q}n_1 }{p^{\ell-\ell_1}}  \right) \notag  \\  
	&\times \sum_{u_{2} \; {\rm mod} \; p^{(\ell-\ell_1)/2}} \overline{\chi} \left(1-\overline{\left(v-p^{r- \ell+\ell_1} u_{1}\right)} p^{r - (\ell- \ell_1)/2} u_{2}\right)e\left(\frac{-\overline{u_1}^2u_2\alpha \overline{q }n_1}{p^{(\ell-\ell_1)/2}}\right).
\end{align*}
We observe that  $u_{2} \mapsto \bar{\chi}\left(1-p^{r-(\ell-\ell_1)/2} u_{2}\right)$
is an additive character modulo $p^{(\ell-\ell_1) /2}$ and of order $p^{(\ell-\ell_1) /2}$.  Thus the sum over $u_2$ can be written as
\begin{align*}
	\sum_{u_{2} \; {\rm mod} \; p^{(\ell-\ell_1)/2}}e\left(\frac{B \overline{\left(v-p^{r- \ell+\ell_1} u_{1}\right)} \; u_{2}}{p^{(\ell-\ell_1) /2}}\right) e\left(\frac{-\overline{u_1}^2u_2\alpha \overline{q }n_1}{p^{(\ell-\ell_1)/2}}\right),
\end{align*}
where $B$ is an absolute constant such that $(B,p)=1$ which depends on $\chi$ only.  Evaluating  the above sum,  we get  the following congruence relation
\begin{align}\label{h_3}
	h_3(v,u_1,\alpha):={q } B \overline{\left(v-p^{r- \ell+\ell_1} u_{1}\right)} - \bar{u_{1}}^{2} n_1\alpha  \equiv \; 0 \; {\rm mod} \; p^{(\ell-\ell_1)/2}, 
\end{align}
along with  the factor $p^{(\ell-\ell_1)/2}$. We observe that if $\ell_4>0$, where  $n_1=p^{\ell_4} n_1^\prime$, then the above equation has no solutions. Hence $\mathcal{C}_u(...)=0$. In the other case, i.e., for   $(n_1,\, p)=1$, or $\ell_4=0$ and $n_1^\prime= n_1$,  we get
\begin{align} \label{sum over u}
	\mathcal{C}_u(...)=p^{(\ell-\ell_1)/2}\sideset{}{^\star} \sum_{\substack{u \, {\rm mod} \, p^{(\ell-\ell_1)/2} \\ h_2(v,u,m) \equiv \; 0 \; {\rm mod} \; p^{r/2}  \\ h_3(v,\, u, \, \alpha) \equiv \; 0 \; {\rm mod} \; p^{(\ell-\ell_1)/2}}} 
	\overline{\chi}\left(v-p^{r-\ell+\ell_1} u\right)e\left(\frac{-\overline{u_1}^2u_2\alpha \overline{q }n_1}{p^{(\ell-\ell_1)/2}}\right).
\end{align}
On plugging the above expression into \eqref{character sum (l prime =0) after a sum }, we get the first part of the lemma. 
\subparagraph*{Case 2} $q=q^\prime p^{\ell^\prime}$, with $\ell^\prime>0$. \\
In this case, we first observe that $(a+bq,\, p)=1$, as $(a,q^\prime p^{\ell^\prime})=1$ and hence $\ell_1=0$. Thus,  on  splitting the sum over  $a$ in \eqref{C before cauchy} as 
	$$a =  a_1 p^{\ell^\prime}\overline{p^{\ell^\prime}}+a_2q^\prime\overline{q^\prime} , \quad \   \; a_1 \; {\rm mod} \; q^\prime \quad \text{and} \quad  a_2 \; {\rm mod} \; p^{\ell^\prime},$$
	we arrive at
		\begin{align}
		\mathcal{C}(...)&=\sideset{}{^\star} \sum_{\alpha \, {\mathrm{mod}}\, p^{\ell_5} q^\prime k^\prime/n_1^\prime}e\left(\frac{\pm n_2\overline{\alpha}}{p^{\ell_5}q^\prime k^\prime /n_1^\prime}\right)  \sideset{}{^\star}\sum_{a_1 \, {\mathrm{mod}}\, q^\prime}   e\left(\frac{\alpha n_1 \overline{a_1 }\overline{p^{\ell+\ell^\prime}}}{ q^\prime }\right) e\left(\frac{\overline{(p^{2r-2\ell_2-\ell+\ell^\prime}a_1)} m}{ q^\prime}\right) \notag \\
		&\times\sideset{}{^\star}\sum_{a_2 \, {\mathrm{mod}}\, p^{\ell^\prime}}   \sum_{b \, {\mathrm{mod}}\, p^{\ell}}   e\left(\frac{\alpha n_1 \overline{((a+bq))} \overline{{q^\prime  }} \, }{p^{\ell+\ell^\prime}  }\right)	\sideset{}{^\star}\sum_{\substack{\beta \, {\rm mod} \, p^{r} }} \bar{\chi}\left(-\beta\right) e\left(\frac{\overline{(c/p^{\ell_2})} m\overline{q^\prime}}{p^{r-\ell_2+\ell^\prime} }\right).
	\end{align}
	 On combining the sum $a_2$ and  the sum over $b$,  we get the following expression for $\mathcal{C}(...)$:
	\begin{align}
		&\sideset{}{^\star} \sum_{\alpha \, {\mathrm{mod}}\, p^{\ell_5} q^\prime k^\prime/n_1^\prime}e\left(\frac{\pm n_2\overline{\alpha}}{p^{\ell_5}q^\prime k^\prime /n_1^\prime}\right)  \sideset{}{^\star}\sum_{a_1 \, {\mathrm{mod}}\, q^\prime}   e\left(\frac{\overline{a_1 }\overline{p^{\ell^\prime}}(\alpha n_1\overline{p^{2\ell}}+\overline{p^{2r-2\ell_2}}m)}{ q^\prime }\right) \notag \\
		&\times  \sideset{}{^\star}\sum_{b \, {\mathrm{mod}}\, p^{\ell+\ell^\prime}}   e\left(\frac{\alpha n_1\overline{((a_1p^{\ell^\prime}\overline{p^{\ell^\prime}}+bq^\prime)) }\overline{q^\prime} }{p^{\ell+\ell^\prime}}  \right) \notag \\
		 &\times \sideset{}{^\star}\sum_{\substack{\beta \, {\rm mod} \, p^{r} }} \bar{\chi}\left(-\beta\right) e\left(\frac{\overline{((p^{r-\ell}(a_1p^{\ell^\prime}\overline{p^{\ell^\prime}}+bq^\prime)+q\beta)/p^{\ell_2})} m\overline{q^\prime}}{p^{r-\ell_2+\ell^\prime} }\right).
	\end{align}
	Now changing the variable $b \mapsto a_1p^{\ell^\prime}\overline{p^{\ell^\prime}}+bq^\prime = u $, we arrive at 
	\begin{align}\label{character for nonprime}
		&\sideset{}{^\star} \sum_{\alpha \, {\mathrm{mod}}\, p^{\ell_5} q^\prime k^\prime/n_1^\prime}e\left(\frac{\pm n_2\overline{\alpha}}{p^{\ell_5}q^\prime k^\prime /n_1^\prime}\right)  \sideset{}{^\star}\sum_{a_1 \, {\mathrm{mod}}\, q^\prime}   e\left(\frac{\overline{a_1 }\overline{p^{\ell^\prime}}(\alpha n_1\overline{p^{2\ell}}+\overline{p^{2r-2\ell_2}}m)}{ q^\prime }\right) \notag \\
		&\times  \sideset{}{^\star}\sum_{u \, {\mathrm{mod}}\, p^{\ell+\ell^\prime}}   e\left(\frac{\alpha n_1\overline{u}\overline{q^\prime} }{p^{\ell+\ell^\prime}}  \right)\
		 \sideset{}{^\star}\sum_{\substack{\beta \, {\rm mod} \, p^{r} }} \bar{\chi}\left(-\beta\right) e\left(\frac{\overline{((p^{r-\ell}u+q\beta)/p^{\ell_2})} m\overline{q^\prime}}{p^{r-\ell_2+\ell^\prime} }\right).
	\end{align}
\subparagraph*{Subcase 2.1} $0< \ell^\prime< r-\ell$. \\
In this situation, we observe that $\ell_2=\ell^\prime$. Thus the sum over $\beta$ becomes
\begin{align*}
	\mathcal{C}_{\beta}(\ell^\prime>0)&:= \sideset{}{^\star}\sum_{\substack{\beta \, {\rm mod} \, p^{r} }} \bar{\chi}\left(-\beta\right) e\left(\frac{\overline{((p^{r-\ell}u+q\beta)/p^{\ell_2})} m\overline{q^\prime}}{p^{r-\ell_2+\ell^\prime} }\right) \\
	&= \sideset{}{^\star}\sum_{\substack{\beta \, {\rm mod} \, p^{r} }} \bar{\chi}\left(-\beta\right) e\left(\frac{\overline{((p^{r-\ell-\ell^\prime}u+q^\prime\beta))} m\overline{q^\prime}}{p^{r} }\right).
\end{align*}
This sum is similar to the sum over $\beta$  in Case 1. Thus analyzing  it in a similar way, we arrive at 
\begin{align}
	\mathcal{C}_{\beta}(\ell^\prime>0)=p^{r/2} \chi(-q^\prime)\sideset{}{^\star} \sum_{\substack{v \, {\rm mod} \, p^{r/2} \\ h_2^\prime(v,u,m) \equiv \; 0 \; {\rm mod} \; p^{r/2} }} \overline{\chi}\left(v-p^{r-\ell-\ell^\prime}u\right) e\left(\frac{\overline{v} m\overline{q^\prime}}{p^{r} }\right),
\end{align}
where $$h_2^\prime(v,u,m)=Aq^\prime v^2 +mv-m p^{r-\ell-\ell^\prime}u.$$ 
On analyzing the sum over $a_1$ and the sum over $u$ as  in Case 1, we get the following expression for $\mathcal{C}(...)$:
	\begin{align} \label{char for non primes }
	&\mathcal{C}(...)=p^{(r+\ell+\ell^\prime)/2}\chi(-q^\prime)\sum_{d \vert q^\prime} d \mu \left(\frac{q^\prime}{d}\right)  \sideset{}{^\star} \sum_{\substack{\alpha \; {\rm mod} \; p^{\ell_5}q^\prime k^\prime/n_1^\prime \\ h_1(\alpha,m)\equiv \; 0 \; {\rm mod} \; d}} e\left(\frac{\pm n_2\overline{\alpha} }{p^{\ell_5}q^\prime k^\prime/n_1^\prime}\right)  \notag  \\
	&\times \mathop{\sideset{}{^\star} \sum_{u \, {\rm mod} \, p^{(\ell+\ell^\prime)/2}} \ \sideset{}{^\star} \sum_{v\, {\rm mod} \, p^{r/2}}}_{\substack { h_2(v, u,m) \equiv \; 0 \; {\rm mod} \; p^{r/2}  \\ h_3(v,\, u, \, \alpha) \equiv \; 0 \; {\rm mod} \; p^{(\ell+\ell^\prime)/2}}} 
	\overline{\chi}\left(v-p^{r-\ell-\ell^\prime} u\right)e\left(\frac{\alpha \overline{u} \overline{q^\prime}n_1 }{p^{\ell+\ell^\prime}}  \right) e\left(\frac{\overline{v} m\overline{q^\prime}}{p^{r} }\right),
\end{align}
where $h_1(\alpha,m)=n_1 \alpha \overline{p^{2\ell+2\ell^\prime}}  + m \overline{p^{2r}}$ and 
$$h_3(v,u,\alpha):={q^\prime } B \overline{\left(v-p^{r- \ell-\ell^\prime} u\right)} - \bar{u}^{2} n_1\alpha  \equiv \; 0 \; {\rm mod} \; p^{(\ell+\ell^\prime)/2}.$$ 
\subparagraph*{Subcase 2.2}$r-\ell <\ell^\prime$.\\
In the case, we have $\ell_2=r-\ell$.
Thus the sum over $\beta$ in \eqref{character for nonprime} is given by 
\begin{align}
	\mathcal{C}_{\beta}(\ell^\prime>r-\ell)=	\sideset{}{^\star}\sum_{\substack{\beta \, {\rm mod} \, p^{r} }} \bar{\chi}\left(-\beta\right) e\left(\frac{\overline{(u+q^\prime\beta p^{\ell^\prime-r+\ell})} m\overline{q^\prime}}{p^{\ell+\ell^\prime} }\right).
\end{align}
We first assume that $(m,p)=1$. On extending this as a sum modulo $p^{\ell+\ell^\prime}$, we see that 
\begin{align}
	\mathcal{C}_{\beta}(\ell^\prime>r-\ell)=	\frac{1}{p^{\ell+\ell_1-r}}\sideset{}{^\star}\sum_{\substack{v \, {\rm mod} \, p^{\ell+\ell^\prime} }} \bar{\chi_1}\left(-v\right) e\left(\frac{\overline{(u+q^\prime v p^{\ell^\prime-r+\ell})} m\overline{q^\prime}}{p^{\ell+\ell^\prime} }\right), 
\end{align}
where $\chi_1$ is the character modulo $p^{\ell+\ell^\prime}$ which is induced from $\chi$. The above equality follows upon writing $v=\beta+v_1p^{r}$ and realizing  the sum over $v_1$ as a free sum.  Now,  upon changing the variable $v \mapsto vu$, we arrive at
\begin{align}
	\mathcal{C}_{\beta}(\ell^\prime>r-\ell)=	\frac{\bar{\chi_1}\left(-u\right)}{p^{\ell+\ell_1-r}} \sideset{}{^\star}\sum_{\substack{v \, {\rm mod} \, p^{\ell+\ell^\prime} }} \bar{\chi_1}\left(v\right) e\left(\frac{\overline{v}\theta }{p^{\ell+\ell^\prime} }\right), 
\end{align}
where $\theta=\overline{(u+q^\prime v p^{\ell^\prime-r+\ell})}m\overline{q^\prime}$. Note that the sum over $v$ can be written as 
$$\sideset{}{^\star}\sum_{\substack{v \, {\rm mod} \, p^{\ell+\ell^\prime} }} \bar{\chi_1}\left(v\right) e\left(\frac{\overline{v}\theta }{p^{\ell+\ell^\prime} }\right)=\sum_{\substack{v \, {\rm mod} \, p^{\ell+\ell^\prime} }} {\chi_1}\left(v\right) e\left(\frac{{v}\theta }{p^{\ell+\ell^\prime} }\right)=\bar{\chi_1}(\theta)\tau(\chi),$$
as $(\theta,p)=1$. We observe that $\tau(\chi_1)=0$, as $\chi_1$ is an imprimitive character modulo $p^{\ell+\ell^\prime}$. See Davenport, [\cite{Davenport}, Chapter 9]. Hence, the character sum vanishes. When $(m,p)>1$, we can carry out  the same analysis by extracting powers of $p$ from $m$, yielding the same result. 
\end{proof}
\begin{remark}\label{coprime q}
	We note  that the expressions of $\mathcal{C}(...)$ for $(q,p)=1$ and $(q,p)>1$ in the above lemma are structurally similar.  So their further analysis will be along the simillar lines. For simplicity,  we will continue with the   expression \eqref{character sum  before cauchy } of $\mathcal{C}(..)$. 
\end{remark}
	\section{Applying  Cauchy and Poisson}\label{cauchy and poisson}
In this section,  we  apply the Cauchy's inequality  followed by the  Poisson summation formula to the sum over $n_2$  in   \eqref{SN before cauchy}.  The aim to apply Cauchy is to get rid of $GL(3)$ Fourier coefficients $A(n_1,n_2)$.  
\subsection{Cauchy's inequality}  
Splitting the sum over $q$ in dyadic blocks $q \sim C$, with $C \ll Q$,  and writing  $q=p^{\ell^\prime}q^\prime=p^{\ell^\prime}q_1^\prime q_2^\prime$ with $(q^\prime, p)=1$, $q_1^\prime \vert (n_1^\prime k^\prime)^\infty$ and  $(n_1^\prime k^\prime, \, q_2^\prime)=1$, we see that $S_k(N)$ in \eqref{SN before cauchy} is dominated by
\begin{align*}
	\sup_{C \ll Q}\frac{p^{\ell_1/2+\ell_2/2}N^{5/4}}{ Q p^{3\ell/2+r} k^{1/2}C^{2}}&  \sum_{\pm} \sum_{\frac{n_1^\prime}{(n_1^\prime ,\, k^\prime)} \ll C p^{\ell^\prime}}   \sum_{\frac{n_1^\prime}{(n_1^\prime, \, k^{\prime})}\vert q_1^\prime \vert (n_1^\prime k^\prime)^\infty} \sum_{n_{2} \ll N_0/n_1^2} \frac{|A(n_{1},n_{2})|}{n_{2}^{1/2}}  \notag   \\
	& \times \bigg | \sum_{q_2^\prime \sim C/(p^{\ell^\prime} q_1^\prime)} \sum_{ m \ll M_0} \frac{\lambda_{f}(m)}{m^{1/4}}\mathcal{C}(...) 	\mathcal{J}\left(n_1^2n_2,p^{\ell^\prime}q^\prime,m\right)\bigg |.
\end{align*}
On splitting the sum over $m$ into dyadic blocks $m \sim M_1$, $M_1 \ll M_0$, and applying the Cauchy's inequality to the sum over $n_2$, we arrive at 
\begin{align}\label{SN after cauchy}
	S_k(N) \ll  \mathop{\sup}_{\substack{C \ll Q \\ M_1 \ll M_0}} 	\frac{p^{\ell_1/2+\ell_2/2}N^{5/4}}{ Q p^{3\ell/2+r} k^{1/2}C^{2}}  \sum_{\pm} \sum_{\frac{n_1^\prime}{(n_1^\prime ,\, k^\prime)} \ll C p^{\ell^\prime}} \Theta^{1/2} \sum_{\frac{n_1^\prime}{(n_1^\prime, \, k^{\prime})}\vert q_1^\prime \vert (n_1^\prime k^\prime)^\infty} \Omega^{1/2},
\end{align}
where 
\begin{align}\label{theta}
	\Theta=\sum_{n_2\ll N_0/n_1^2} \frac{|A(n_1,n_2)|^2}{n_2},
\end{align} 
and
\begin{align}\label{Omega }
	\Omega = \sum_{n_{2} \ll N_{0}/n_{1}^2}\; \Big| \sum_{\substack{q_2 ^\prime \sim C/(p^{\ell^\prime }q_1^\prime) }} \sum_{m \sim  M_{1}} \frac{\lambda_f(m)}{m^{1/4}} \; \mathcal{C} \left(...\right) \mathcal{J}\left(n_{1}^2 n_{2}, p^{\ell^\prime}q^\prime,m\right)\Big|^{2},
\end{align}
with  $$M_{1} \leq M_{0}=p^{2r-\ell-2\ell_2}R^{\epsilon} \quad \text{and } \; N_{0} = N^{1/2}p^{3\ell/2-3\ell_1}kR^{\epsilon}.$$
\subsection{Poisson summation}\label{pois}
In this step, we will apply the Poisson summation formula to the sum over $n_2$ in \eqref{Omega }. To this end, we first smooth out the sum over $n_2$ using a smooth bump function $V$. In fact,  splitting  $n_2$ into dyadic blocks $n_2 \sim N^\prime$, $N^\prime \ll N_0/n_1^2$, we   arrive at the following expression:
\begin{align}\label{Omega after smoothing}
	\Omega \ll \sup_{N^\prime \ll N_0/n_1^2} \mathop{\sum \sum}_{q_2^\prime, q_2^{\prime \prime} \sim C/(p^{\ell^\prime}q_1^\prime)} \mathop{\sum \sum}_{m, \, m^\prime \sim M_{1}}\frac{ |\lambda_f(m) \lambda_f(m^\prime)|}{(mm^\prime)^{1/4}} |L\left(...\right)|,
\end{align}
where 
\begin{align}\label{L before poisson} 
	L(...)=\sum_{n_{2}\in \mathbb{Z}} V\left(\frac{ n_{2}}{N^\prime}\right) \mathcal{C} \left(...\right) \overline{\mathcal{C} \left(...\right)} \; \mathcal{J}\left(n_{1}^2 n_{2}, p^{\ell^\prime}q^\prime,m\right) \overline{\mathcal{J}\left(n_{1}^2 n_{2}, p^{\ell^\prime}q^{\prime \prime},m^\prime\right)},
\end{align}
$q^{\prime \prime} =q_1^\prime q_2^{\prime \prime}$ and $V$ is a smooth bump function supported on $[1, \, 2]$. Reducing $n_2$  modulo $\mathcal{Q}:={p^{\ell_5} q_1^\prime q_2^\prime q_2^{\prime \prime } k^\prime }/{n_1^\prime}$, i.e., changing the variable  $n_2 \mapsto \nu +n_2\mathcal{Q}$ in the expression of $L(...)$, we arrive at 
\begin{align*}
	&\sum_{\nu \, {\rm mod} \, \mathcal{Q} }\mathcal{C} \left(...\right) \overline{\mathcal{C} \left(...\right)} 
	\sum_{n_{2}\in \mathbb{Z}} V\left(\frac{(n_{2} \mathcal{Q}+\nu)}{N^\prime}\right)    \\
	&\times \mathcal{J}\left(n_{1}^2 (n_{2}\mathcal{Q} +\nu), p^{\ell^\prime}q^\prime,m\right) \overline{\mathcal{J}\left(n_{1}^2 (n_{2}\mathcal{Q} +\nu), p^{\ell^\prime}q^{\prime \prime},m^\prime\right)}.
\end{align*}
Now on applying the Poisson summation formula to the sum over $n_2$, we get 
\begin{align}\label{L after poisson}
	L(...)=\frac{N^\prime}{ \mathcal{Q} } \sum_{n_{2} \in \mathbb{Z}}  \sum_{\nu \, {\rm mod} \; \mathcal{Q} }\mathcal{C} \left(...\right) \overline{\mathcal{C} \left(...\right)}e \left(\frac{\nu n_{2}}{\mathcal{Q} }\right) \mathcal{I} \left(....\right),
\end{align}
where the integral $\mathcal{I}(...)$ is given by
\begin{align}\label{integral after poisson}
	\mathcal{I}(...)=\int_{\mathbb{R}} V(w) \mathcal{J}\left(n_1^2 N^\prime w, p^{\ell^\prime}q^{\prime},m\right) \overline{\mathcal{J}\left(n_1^2 N^\prime w,p^{\ell^\prime} q^{\prime \prime},m^\prime \right)} \; e\left(\frac{-n_{2}N^\prime w}{\mathcal{Q}  }\right) \mathrm{d}w.
\end{align} 
Finally, on plugging \eqref{L after poisson} in \eqref{Omega after smoothing}, we get
\begin{align}\label{final Omega}
	\Omega \ll \frac{ 1 }{ M_{1}^{1/2}} \sup_{N^\prime \ll N_0/n_1^2} N^\prime \mathop{\sum \sum}_{q_2^\prime, q_2^{\prime \prime} \sim C/(p^{\ell^\prime}q_1^\prime)} \mathop{\sum \sum}_{m,m' \sim M_{1}} \sum_{n_{2} \in \mathbb{Z}} \left|\mathfrak{C}(...)\right| \left|\mathcal{I}(...) \right|,
\end{align}
where 
\begin{align}\label{sum over beta}
	\mathfrak{C}(...):=\frac{1}{\mathcal{Q}} \sum_{\nu \, {\rm mod} \; \mathcal{Q} }\mathcal{C} \left(...\right) \overline{\mathcal{C} \left(...\right)}e \left(\frac{\nu  n_{2}}{\mathcal{Q} }\right).
\end{align}
	\subsection{ The sum over $\nu$} Following Remark \ref{coprime q}, we   assume that $(q,p)=1$. Thus $\ell^\prime=0$ and  $q=q^\prime$. Also, we can take $n_1=n_1^\prime$, as otherwise, the character sum in Lemma \eqref{characer for r even} vanishes.  On plugging the expression of $\mathcal{C}(...)$ from   \eqref{character sum  before cauchy } into \eqref{sum over beta}, we see that the sum over $\nu$ is given by 
\begin{align*}
\frac{1}{\mathcal{Q}}\sum_{\nu \, {\rm mod} \; \mathcal{Q} } e\left(\frac{\pm \nu \overline{\alpha} }{p^{\ell_5}q_1^\prime q_2^\prime k^\prime/n_1^\prime}\right)  e\left(\frac{\mp \nu \overline{\alpha^\prime} }{p^{\ell_5}q_1^\prime q_2^{\prime \prime} k^\prime/n_1^\prime}\right)e \left(\frac{\nu  n_{2}}{\mathcal{Q} }\right).
\end{align*}
It's evaluation gives us the following congruence relation:
$$\pm \overline{\alpha}q_2^{\prime \prime} \mp  \overline{\alpha^\prime}q_2^{ \prime}+n_2 \equiv\, 0 \, \mathrm{mod} \, \mathcal{Q}.$$
Thus we are left  with the following expression of $\mathfrak{C}(...)$:
\begin{align} \label{character after poisson}
	&\chi(q q^{\prime \prime}) p^{r+\ell-\ell_1}  \mathop{\sum \sum}_{\substack{d \vert q \\ d^\prime \vert q^{\prime \prime}}} dd^\prime \mu \left(\frac{q}{d}\right) \mu \left(\frac{q^{\prime \prime }}{d^\prime}\right)  \mathop{\sideset{}{^\star} \sum_{\alpha \; {\rm mod} \; p^{\ell_5}qk^\prime/n_1^\prime } \  \sideset{}{^\star} \sum_{\alpha^\prime \; {\rm mod} \; p^{\ell_5}q^{\prime \prime }k^\prime/n_1^\prime} }_{\substack{ h_1(\alpha, \, m)\equiv \; 0 \; {\rm mod} \; d \\h_1(\alpha^\prime, \, m^\prime) \equiv \; 0 \; {\rm mod} \; d^{\prime} \\\pm \overline{\alpha}q_2^{\prime \prime} \mp  \overline{\alpha^\prime}q_2^{ \prime}+n_2 \equiv\, 0 \, \mathrm{mod} \, \mathcal{Q} }}  \notag \\
	&\times \mathop{\sideset{}{^\star} \sum_{u \, {\rm mod} \, p^{(\ell-\ell_1)/2}} \ \sideset{}{^\star} \sum_{v\, {\rm mod} \, p^{r/2}}}_{\substack { h_2(v, u,m) \equiv \; 0 \; {\rm mod} \; p^{r/2}  \\ h_3(v,\, u, \, \alpha) \equiv \; 0 \; {\rm mod} \; p^{(\ell-\ell_1)/2}}} 
	\overline{\chi}\left(v-p^{r-\ell+\ell_1} u\right) e\left(\frac{\alpha \overline{u} \overline{q}n_1 }{p^{\ell-\ell_1}}  \right) e\left(\frac{\overline{v} m\overline{q}}{p^{r}  }\right) \notag  \\
	&\times \mathop{\sideset{}{^\star} \sum_{u^\prime \, {\rm mod} \, p^{(\ell-\ell_1)/2}} \ \sideset{}{^\star} \sum_{v^\prime \, {\rm mod} \, p^{r/2}}}_{\substack { h_2(v^\prime, u^\prime,m^\prime) \equiv \; 0 \; {\rm mod} \; p^{r/2}  \\ h_3(v^\prime ,\, u^\prime, \, \alpha^\prime) \equiv \; 0 \; {\rm mod} \; p^{(\ell-\ell_1)/2}}} 
	{\chi}\left(v^\prime-p^{r-\ell+\ell_1} u^\prime\right)e\left(\frac{-\overline{u^\prime}\alpha^\prime  \overline{q^{\prime \prime} }n_1}{p^{\ell-\ell_1}}\right)e\left(\frac{-\overline{v^\prime} m^\prime \overline{q^{\prime \prime}}}{p^{r} }\right). 
\end{align}
\section{Final estimates for the character sum $\mathcal{\mathfrak{C}(...)}$}\label{rest cancellations for c}
In this section, we will give final estimates for the character sum $\mathfrak{C}(...)$ given in \eqref{character after poisson}. 
\begin{lemma}\label{final character bound}
 Let $\mathfrak{C}_0(...)$  and $\mathfrak{C}_{\neq 0}(...)$  denote the contributions of $n_2=0$ and $n_2 \neq 0$ respectively  to $\mathfrak{C}(...)$  in \eqref{character after poisson}. Then we have 
 \begin{align*}
 	\mathfrak{C}_0(...) \ll  p^{r+2(\ell-\ell_1)} & \mathop{\sum \sum}_{\substack{d,d^\prime  \vert q \\ (d,d^\prime) \vert \frac{(m-m^\prime)}{p^{(r-\ell+\ell_1)}}}} dd^\prime  \frac{ q k }{[d,d^\prime]},
 	 \end{align*}
 and 
 $$	\mathfrak{C}_{\neq 0}(...) \ll \frac{{q_1^{\prime  2} } k (m,n_1^\prime) }{n_1^\prime}p^{r+3(\ell-\ell_1)/2}\mathop{\sum \sum}_{\substack{d_2  \vert (q_2^\prime, -n_2m \pm q_2^{\prime \prime} n_1^\prime p^{\ell_6} ) \\  d_2^\prime  \vert (q_2^{\prime \prime}, \,  n_2m^\prime  \pm q_2^{ \prime} n_1^\prime p^{\ell_6} )  }}d_2d_2^\prime,$$
 \end{lemma}
where $\ell_6=2(r-\ell+\ell_1)$. 
\begin{proof}
 Splitting $\alpha$ and $\alpha^\prime$ as 
	\begin{align*}
		\alpha&= \alpha_1  \overline{p^{\ell_5}}p^{\ell_5}+\alpha_2 \overline{qk^\prime/n_1^\prime} qk^\prime/n_1^\prime, \quad  \alpha_1 \ \text{mod} \ qk^\prime /n_1^\prime  \ \text{and} \  \alpha_2 \ \mathrm{mod}\ p^{\ell_5}, \\
		 \alpha^\prime&= \alpha_1^\prime  \overline{p^{\ell_5}}p^{\ell_5}+\alpha_2^\prime \overline{q^{\prime \prime }k^{\prime  }/n_1^\prime} q^{\prime \prime }k^\prime/n_1^\prime, \quad  \alpha_1^\prime  \ \text{mod} \ q^{\prime \prime }k^\prime /n_1^\prime  \ \text{and} \  \alpha_2^\prime \ \mathrm{mod}\ p^{\ell_5},
	\end{align*}
		we observe that the sum over $\alpha_2$ and $\alpha_2^\prime$ is given by 
		\begin{align*}
		\mathfrak{C}_{\alpha_2,\alpha_2^\prime}(...):=	\mathop{\sideset{}{^\star} \sum_{\alpha_2 \; {\rm mod} \; p^{\ell_5} } \  \sideset{}{^\star} \sum_{\alpha_2^\prime \; {\rm mod} \; p^{\ell_5}} }_{\substack{ h_3(v,u, \alpha_2)\equiv \; 0 \; {\rm mod} \; p^{(\ell-\ell_1)/2} \\  h_3(v^\prime,u^\prime, \alpha_2^\prime)\equiv \; 0 \; {\rm mod} \; p^{(\ell-\ell_1)/2}  \\ \pm \overline{\alpha_2}q_2^{\prime \prime} \mp  \overline{\alpha_2^\prime}q_2^{ \prime}+n_2 \equiv\, 0 \, \mathrm{mod} \, p^{\ell_5} }} e\left(\frac{\overline{u}\alpha_2  \overline{q  }n_1}{p^{\ell-\ell_1}}\right) e\left(\frac{-\overline{u^\prime}\alpha_2^\prime \overline{q^{\prime \prime}}n_1 }{p^{\ell-\ell_1}}\right),
		\end{align*}
	and the sum over $\alpha_1$ and $\alpha_1^\prime$ is given by 
	\begin{align*}
	\mathfrak{C}_{\alpha_1,\alpha_1^\prime}(...):=	\mathop{\sideset{}{^\star} \sum_{\alpha_1 \; {\rm mod} \; qk^\prime/n_1^\prime } \  \sideset{}{^\star} \sum_{\alpha_1^\prime \; {\rm mod} \; q^{\prime \prime }k^\prime/n_1^\prime} }_{\substack{ h_1(\alpha_1, \, m)\equiv \; 0 \; {\rm mod} \; d \\h_1(\alpha_1^\prime, \, m^\prime) \equiv \; 0 \; {\rm mod} \; d^{\prime} \\ \pm \overline{\alpha_1}q_2^{\prime \prime} \mp  \overline{\alpha_1^\prime}q_2^{ \prime}+n_2 \equiv\, 0 \, \mathrm{mod} \,q_1^\prime q_2^\prime q_2^{\prime \prime}k^\prime /n_1^\prime }}1,
	\end{align*}
which we will analyze further. 
\subparagraph*{Case 1} Let's first assume that  $n_2=0$.  In this case,  the congruence 
$$\pm \overline{\alpha_1}q_2^{\prime \prime} \mp  \overline{\alpha_1^\prime}q_2^{ \prime}+n_2 \equiv\, 0 \, \mathrm{mod} \,q_1^\prime q_2^\prime q_2^{\prime \prime}k^\prime /n_1^\prime$$
implies that $q_2^\prime =q_2^{\prime \prime }$ and hence  $\alpha_1 \equiv \alpha_1^\prime \, \mathrm{mod}\, q k^\prime /n_1^\prime $. Furthermore, 
 the congruence 
 $$\pm \overline{\alpha_2}q_2^{\prime \prime} \mp  \overline{\alpha_2^\prime}q_2^{ \prime}+n_2 \equiv\, 0 \, \mathrm{mod} \, p^{\ell_5}$$
 yields  $\alpha_2 \equiv \alpha_2^\prime \, \mathrm{mod}\, p^{\ell_5}$. Thus  
 $	\mathfrak{C}_{\alpha_2,\alpha_2^\prime}(...)$ transforms into
 	\begin{align*}
 	\mathfrak{C}_{\alpha_2,\alpha_2^\prime}(...)=	\mathop{\sideset{}{^\star} \sum_{\alpha_2 \; {\rm mod} \; p^{\ell_5} }  }_{\substack{ h_3(v,u, \alpha_2)\equiv \; 0 \; {\rm mod} \; p^{(\ell-\ell_1)/2} \\  h_3(v^\prime,u^\prime, \alpha_2)\equiv \; 0 \; {\rm mod} \; p^{(\ell-\ell_1)/2}  }} e\left(\frac{(\overline{u}-\overline{u^\prime})\alpha_2  \overline{q }n_1 }{p^{\ell-\ell_1}}\right).
 \end{align*}
 Recall that $\ell_5=\ell-\ell_1+\ell_3$. On reducing $\alpha_2$ modulo $p^{ (\ell-\ell_1)/2}$, 
  we get 
 	\begin{align*}
 	\mathfrak{C}_{\alpha_2,\alpha_2^\prime}(...)= 	p^{\ell_3+(\ell-\ell_1)/2 } \mathop{\sideset{}{^\star} \sum_{\alpha_2 \; {\rm mod} \; p^{(\ell-\ell_1)/2} }  }_{\substack{ h_3(v,u, \alpha_2)\equiv \; 0 \; {\rm mod} \; p^{(\ell-\ell_1)/2} \\  h_3(v^\prime,u^\prime, \alpha_2)\equiv \; 0 \; {\rm mod} \; p^{(\ell-\ell_1)/2}  }} e\left(\frac{(\overline{u}-\overline{u^\prime})\alpha_2  \overline{q }n_1 }{p^{\ell-\ell_1}}\right),
 \end{align*}
 along with the congruence $(\overline{u}-\overline{u^\prime})  \overline{q} p^{\ell_3}n_1 \equiv 0\, \mathrm{mod}\, p^{\ell_3+(\ell-\ell_1)/2}$ which implies that $u \equiv u^\prime \, \mathrm{mod}\,  p^{(\ell-\ell_1)/2}$.  Hence, 
 $$h_3(v,u,\alpha_2)-h_3(v^\prime, u^\prime,\alpha_2 ) \equiv {q^\prime } B \overline{\left(v-p^{r- \ell+\ell_1} u\right)} -{q^\prime } B \overline{\left(v^\prime-p^{r- \ell+\ell_1} u\right)}  \equiv \,   0 \; {\rm mod} \; p^{(\ell-\ell_1)/2} ,$$
 giving us $v \equiv v^\prime\, \mathrm{mod} \, p^{(\ell-\ell_1)/2}$ and consequently, 
 $$h_2(v,u,m)-h_2(v^\prime, u^\prime,m^\prime) \equiv (m-m^\prime)v \equiv 0\, \mathrm{mod}\, p^{r-\ell+\ell_1},$$
 which yields the restriction $p^{r-\ell+\ell_1} \vert (m-m^\prime)$.
   Using $\alpha_1 \equiv \alpha_1^\prime \, \mathrm{mod}\, q k^\prime /n_1^\prime $, we arrive at the following expression of $	\mathfrak{C}_{\alpha_1,\alpha_1^\prime}(...)$:
 	\begin{align*}
 	\mathfrak{C}_{\alpha_1,\alpha_1^\prime}(...)=	\mathop{\sideset{}{^\star} \sum_{\alpha_1 \; {\rm mod} \; qk^\prime/n_1^\prime }  }_{\substack{ h_1(\alpha_1, \, m)\equiv \; 0 \; {\rm mod} \; d \\h_1(\alpha_1, \, m^\prime) \equiv \; 0 \; {\rm mod} \; d^{\prime}  }}1.
 \end{align*}
Hence, combining all the above observations together,  we see that $\mathfrak{C}_0(...)$ is dominated by
\begin{align*} 
	\mathfrak{C}_0(...) \ll  p^{r+\ell-\ell_1}p^{\ell_3+(\ell-\ell_1)/2} & \mathop{\sum \sum}_{\substack{d,d^\prime  \vert q \\ (d,d^\prime) \vert \frac{(m-m^\prime)}{p^{(r-\ell+\ell_1)}}}} dd^\prime  \frac{ q k^\prime }{[d,d^\prime]} \,  \mathop{ \sideset{}{^\star} \sum_{u \, {\rm mod} \, p^{(\ell-\ell_1)/2}} \ \sideset{}{^\star} \sum_{u^\prime \, {\rm mod} \,p^{(\ell-\ell_1)/2} }}_{\substack { u \equiv u^\prime   \; {\rm mod} \; p^{(\ell-\ell_1)/2}}}  \notag \\
	&\times \mathop{\sideset{}{^\star} \sum_{v \, {\rm mod} \, p^{r/2}} \ \sideset{}{^\star} \sum_{v^\prime\, {\rm mod} \, p^{r/2}}}_{\substack { h_2(v, u,m) \equiv \; 0 \; {\rm mod} \; p^{r/2}  \\   h_2(v^\prime, u^\prime,m^\prime) \equiv \; 0 \; {\rm mod} \; p^{r/2} }} \mathop{\sideset{}{^\star} \sum_{\alpha_2 \; {\rm mod} \; p^{(\ell-\ell_1)/2} }  }_{\substack{ h_3(v,u, \alpha_2)\equiv \; 0 \; {\rm mod} \; p^{(\ell-\ell_1)/2}   }}1.
\end{align*}
Now we count the number of tuples  $(u,u^\prime, v, v^\prime, \alpha_2)$. We observe that, given $u$ and $v$, the congruence 
	$$h_3(v,u,\alpha_2)={q^\prime } B \overline{\left(v-p^{r- \ell+\ell_1} u\right)} - \bar{u}^{2} n_1^\prime \alpha_2  \equiv \; 0 \; {\rm mod} \; p^{(\ell-\ell_1)/2}$$
determines $\alpha_2$ uniquely. Next we count the number of  $v$ and $v^\prime$ using the Hensel's lemma.  Let's  consider
\begin{align}
	 h_2(v,u,m)=Aq^\prime v^2 +mv-m p^{r-\ell+\ell_1}u  \equiv \; 0 \; {\rm mod} \; p^{r/2},
\end{align}
 in which we  want to count the number of   $v$'s modulo $p^{r/2}$ (keeping $u$ fixed) satisfying  $h_2(v,u,m)$. We may assume that  $(m,p)=1$, as otherwise $h_2(v,u,m)$ has no solutions. Let $v_0$ be a solution of $h_2(v,u,m)$ modulo $p^{r/2}$.  We observe that  $$v_0 \equiv -m\overline{Aq^\prime} \,  {\rm mod} \; p^{r-\ell+\ell_1}. $$
By the Hensel's lemma, it can be lifted uniquely modulo $p^{r/2}$. Hence  $v_0$ is determined  uniquely modulo $p^{r/2}$. Similar arguments apply to $h_2(v^\prime,u^\prime,m^\prime)$ as well. Hence,  on estimating the sum over $u$ trivially, we arrive at 
\begin{align*}
		\mathfrak{C}_0(...) \ll  p^{r+\ell-\ell_1}p^{\ell_3+(\ell-\ell_1)} & \mathop{\sum \sum}_{\substack{d,d^\prime  \vert q \\ (d,d^\prime) \vert \frac{(m-m^\prime)}{p^{(r-\ell+\ell_1)}}}} dd^\prime  \frac{ q k^\prime }{[d,d^\prime]}. 
\end{align*}
Hence we have the first part of the lemma. 
\subparagraph*{Case 2} Now we will analyze $\mathfrak{C}(...)$ for $n_2 \neq 0$.  The congruence 
$$\pm \overline{\alpha_2}q_2^{\prime \prime} \mp  \overline{\alpha_2^\prime}q_2^{ \prime}+n_2 \equiv\, 0 \, \mathrm{mod} \, p^{(\ell_5=\ell-\ell_1+\ell_3)}$$
determines $\alpha_2^\prime$ in terms of $\alpha_2$. In fact, taking $+$  sign for simplicity, we have 
$$\alpha_2^\prime \equiv q_2^\prime \overline{(n_2+\bar{\alpha_2}q_2^{\prime \prime})} \equiv q_2^\prime \alpha_2 \overline{(n_2\alpha_2+q_2^{\prime \prime})} \ \mathrm{mod} \, p^{\ell-\ell_1+\ell_3}. $$
 Thus,  upon changing the variable $\gamma_2=n_2\alpha_2+q_2^{\prime \prime}$, we get the following expression of $\mathfrak{C}_{\alpha_2, \alpha_2^\prime}(...)$:
 	\begin{align*}
  \mathfrak{C}_{\alpha_2, \alpha_2^\prime}(...)=&e\left(\frac{( -\overline{u}q_2^{\prime \prime } \overline{q_2^{ \prime }} -\overline{u^\prime}q_2^\prime \overline{q_2^{\prime \prime }}) \overline{n_2} \overline{q_1^{ \prime}} n_1 }{p^{\ell_5}}\right)  \\
 		&\times \mathop{\sideset{}{^\star} \sum_{\gamma_2 \; {\rm mod} \; p^{\ell-\ell_1+\ell_3} }}_{\substack{ h_3(v,u, \overline{n_2}(\gamma_2-q_2^{\prime \prime } ) )\equiv \; 0 \; {\rm mod} \; p^{(\ell-\ell_1)/2} \\  h_3(v^\prime,u^\prime, \overline{n_2}(\gamma_2-q_2^{\prime \prime })q_2^\prime \overline{\gamma_2} )\equiv \; 0 \; {\rm mod} \; p^{(\ell-\ell_1)/2}   }} e\left(\frac{(\overline{u}  \overline{q_2^\prime}\gamma_2 +\overline{u^\prime} q_2^\prime \overline{\gamma_2})  \overline{n_2} \overline{q_1^{ \prime} }n_1 }{p^{\ell-\ell_1}}\right).
 \end{align*}
 Note that we have assumed $(n_2,p)=1$. In the other case, on extracting the $p$-powers from $n_2$, we can analyze similarly.   Now reducing $\gamma_2$ modulo $(\ell-\ell_1)/2$, we get   the following sum over $\gamma_2$:
 \begin{align*}
 	p^{(\ell_3+(\ell-\ell_1)/2)} \mathop{\sideset{}{^\star} \sum_{\gamma_2 \; {\rm mod} \; p^{(\ell-\ell_1)/2} }}_{\substack{ h_3(v,u, \overline{n_2}(\gamma_2-q_2^{\prime \prime } ) )\equiv \; 0 \; {\rm mod} \; p^{(\ell-\ell_1)/2} \\  h_3(v^\prime,u^\prime, \overline{n_2}(\gamma_2-q_2^{\prime \prime })q_2^\prime \overline{\gamma_2} )\equiv \; 0 \; {\rm mod} \; p^{(\ell-\ell_1)/2}  \\
 	 h_4(u,u^\prime, \gamma_2) \equiv \; 0 \; {\rm mod} \; p^{(\ell-\ell_1)/2}  }} e\left(\frac{(\overline{u}  \overline{q_2^\prime}\gamma_2 +\overline{u^\prime} q_2^\prime \overline{\gamma_2})  \overline{n_2} \overline{q_1^{ \prime} }n_1 }{p^{\ell-\ell_1}}\right),
 \end{align*}
 where $$h_4(u, u^\prime, \gamma_2):=\overline{u}  \overline{q_2^\prime} -\overline{u^\prime} q_2^\prime \overline{\gamma_2}^2.$$
 Hence, we have 
  $$\mathfrak{C}_{\alpha_2, \alpha_2^\prime}(...) \ll  	p^{(\ell_3+(\ell-\ell_1)/2)} \mathop{\sideset{}{^\star} \sum_{\gamma_2 \; {\rm mod} \; p^{(\ell-\ell_1)/2} }}_{\substack{ h_3(v,u, \overline{n_2}(\gamma_2-q_2^{\prime \prime } ) )\equiv \; 0 \; {\rm mod} \; p^{(\ell-\ell_1)/2} \\  h_3(v^\prime,u^\prime, \overline{n_2}(\gamma_2-q_2^{\prime \prime })q_2^\prime \overline{\gamma_2} )\equiv \; 0 \; {\rm mod} \; p^{(\ell-\ell_1)/2}  \\
  		h_4(u,u^\prime, \gamma_2) \equiv \; 0 \; {\rm mod} \; p^{(\ell-\ell_1)/2}  }}1.$$
On  plugging the above expression into \eqref{character after poisson}, we arrive at 
 \begin{align*}
 	\mathfrak{C}_{\neq 0}(...) \ll p^{r+\ell_3+3(\ell-\ell_1)/2} \,  \, \mathfrak{C}_{v,v^\prime,u,u^\prime, \gamma_2}(...) \mathop{\sum \sum}_{\substack{d \vert q \\ d^\prime \vert q^{\prime \prime}}} dd^\prime\,  \mathfrak{C}_{\alpha_1, \alpha_1^\prime}(...), 
 \end{align*}
 where 
 \begin{align*}
 	\mathfrak{C}_{v,v^\prime,u,u^\prime, \gamma_2}(...):=\mathop{\sideset{}{^\star} \sum_{v \, {\rm mod} \, p^{r/2}} \ \sideset{}{^\star} \sum_{v^\prime\, {\rm mod} \, p^{r/2}}}_{\substack { h_2(v,u,m) \equiv \; 0 \; {\rm mod} \; p^{r/2}  \\   h_2(v^\prime, u^\prime,m^\prime) \equiv \; 0 \; {\rm mod} \; p^{r/2} }} \,  \mathop{ \sideset{}{^\star} \sum_{u \, {\rm mod} \, p^{(\ell-\ell_1)/2}} \ \sideset{}{^\star} \sum_{u^\prime \, {\rm mod} \,p^{(\ell-\ell_1)/2}} \  \sideset{}{^\star} \sum_{\gamma_2 \; {\rm mod} \; p^{(\ell-\ell_1)/2} }}_{\substack{ h_3(v,u, \overline{n_2}(\gamma_2-q_2^{\prime \prime } ) )\equiv \; 0 \; {\rm mod} \; p^{(\ell-\ell_1)/2} \\  h_3(v^\prime,u^\prime, \overline{n_2}(\gamma_2-q_2^{\prime \prime })q_2^\prime \overline{\gamma_2} )\equiv \; 0 \; {\rm mod} \; p^{(\ell-\ell_1)/2}  \\
 			h_4(u,u^\prime, \gamma_2) \equiv \; 0 \; {\rm mod} \; p^{(\ell-\ell_1)/2}  }}1.
 \end{align*}
 Our next step is to analyze $	\mathfrak{C}_{v,v^\prime,u,u^\prime, \gamma_2}(...)$. We will prove that 
 \begin{align}\label{counting equation}
 		\mathfrak{C}_{v,v^\prime,u,u^\prime, \gamma_2}(...) \ll p^{\epsilon}. 
 \end{align} 
We have five variables $v, \,  v^\prime, \,  u,\,  u^\prime$ and $\gamma_2$ with the following five congruences:
 \begin{align}\label{h2}
 h_2(v,u,m)=Aq^\prime v^2 +mv-m p^{r-\ell+\ell_1}u  \equiv \; 0 \; {\rm mod} \; p^{r/2},
 \end{align}
 \begin{align}\label{h2prime}
	h_2(v^\prime,\, u^\prime)=Aq^{\prime \prime} v^{\prime 2} +m^\prime v^\prime-m^\prime p^{r-\ell+\ell_1}u^\prime  \equiv \; 0 \; {\rm mod} \; p^{r/2},
\end{align}

\begin{align}\label{h3}
	h_3(v,\, u,...)={q^\prime } B \overline{\left(v-p^{r- \ell+\ell_1} u\right)} - \bar{u}^{2} n_1^\prime \overline{n_2}(\gamma_2-q_2^{\prime \prime } )  \equiv \; 0 \; {\rm mod} \; p^{(\ell-\ell_1)/2},
\end{align}
 \begin{align}\label{h3prime}
 		h_3(v^\prime, \, u^\prime,... )={q^{\prime \prime} } B \overline{\left(v^\prime-p^{r- \ell+\ell_1} u^\prime \right)} - \bar{u^\prime}^{2} n_1^\prime \overline{n_2}(\gamma_2-q_2^{\prime \prime })q_2^\prime \overline{\gamma_2}   \equiv \; 0 \; {\rm mod} \; p^{(\ell-\ell_1)/2},
 \end{align}
 \begin{align} \label{h4}
 	h_4(u, \,  u^\prime, \gamma_2)=\overline{u}  \overline{q_2^\prime} -\overline{u^\prime} q_2^\prime \overline{\gamma_2}^2 \equiv \; 0 \; {\rm mod} \; p^{(\ell-\ell_1)/2}.
 \end{align}
We observe from  \eqref{h4} that, fixing $u$ and $u^\prime$, $\gamma_2$ has at most two choices. In fact, 
\begin{align}\label{gamma 2}
	\gamma_2^2 \equiv u^\prime \overline{u}{\overline{q_2^\prime}}^2 \; {\rm mod} \; p^{(\ell-\ell_1)/2}.
\end{align}
Let's now consider \eqref{h2}, in which we  want to find the number of solutions (keeping $u$ fixed) of $h_2(v,u,m)$ modulo $p^{r/2}$. We argue as in  the zero frequency case.  Let $v_0$ be a solution of $h_2(v,u,m)$ modulo $p^{r/2}$.  We observe that  $$v_0 \equiv -m\overline{Aq^\prime} \,  {\rm mod} \; p^{r-\ell+\ell_1}. $$
By Hensel's lemma, it can be lifted uniquely modulo $p^{r/2}$. Hence  $v_0$ is determined  uniquely modulo $p^{r/2}$. The same arguments can be applied to \eqref{h2prime}. Thus, $h_2(v^\prime,u^\prime,m^\prime)$ also has a unique  solution, say,  $v_0^\prime$ modulo $p^{r/2}$ such that 
$$v_0^\prime \equiv -m^\prime\overline{Aq^{\prime \prime}} \,  {\rm mod} \; p^{r-\ell+\ell_1}. $$ 
Now it remains to count the number of pairs $(u, u^\prime)$. 
On substituting \eqref{gamma 2} and $v_0-p^{r-\ell+\ell_1}u=-\overline{m}Aq^\prime v_0^2$ into \eqref{h3}, we get
\begin{align}\label{u prime }
	u{q_2^{\prime }}^2 \left(-\overline{n_1^\prime} n_2  Bm\overline{A} \bar{v_0}^2u^2+q_2^{\prime \prime }\right)^2  \equiv  u^\prime  \; {\rm mod} \; p^{(\ell-\ell_1)/2},
\end{align}
 which determines $u^\prime$ uniquely in terms of $u$, as $v_0$ is depending only on $u$.  Now substituting \eqref{gamma 2} in place of  $\gamma_2^2$ and $-\overline{m^\prime}Aq^{\prime \prime}$ in place of  $v_0^\prime-p^{r-\ell+\ell_1}u^\prime $ into \eqref{h3prime}, we arrive at 
 \begin{align}\label{u }
u^\prime (\overline{ q_2^{\prime \prime }q_2^{\prime }})^2 \left(-\overline{n_1^\prime q_2^\prime} n_2  Bm^\prime \overline{A} \overline{v_0^\prime}^2{u^\prime}^2+1\right)^2  \equiv  u  \; {\rm mod} \; p^{(\ell-\ell_1)/2}.
 \end{align}
Reducing \eqref{u prime } and \eqref{u } modulo $p^{r-\ell+\ell_1}$, we get 
\begin{align}\label{u prime reduced}
		u \left(B_2u^2+q_2^\prime q_2^{\prime \prime }\right)^2  \equiv  u^\prime  \; {\rm mod} \; p^{(r-\ell+\ell_1)},
\end{align}
 \begin{align} \label{u reduced}
	u^\prime  \left( B_3{u^\prime}^2+ \overline{ q_2^{\prime \prime }q_2^{\prime }} \right)^2  \equiv  u  \; {\rm mod} \; p^{(r-\ell+ \ell_1)},
\end{align}
where $B_2= -\overline{n_1^\prime} n_2  B\overline{m}A {q^\prime}^2q_2^\prime $ and $B_3=-\overline{n_1^\prime q_2^\prime} n_2  B \overline{m^\prime}A {q^{\prime \prime }}^2\overline{ q_2^{\prime \prime }q_2^{\prime }}$. On plugging \eqref{u prime reduced} into \eqref{u reduced}, we get 
 \begin{align} 
	u \left(B_2u^2+q_2^\prime q_2^{\prime \prime }\right)^2  \left( B_3	u^2 \left(B_2u^2+q_2^\prime q_2^{\prime \prime }\right)^4+ \overline{ q_2^{\prime \prime }q_2^{\prime }} \right)^2  \equiv  u  \; {\rm mod} \; p^{(r-\ell+ \ell_1)}. 
\end{align}
Reducing it modulo $p$, we get 
 \begin{align} 
	\left(B_2u^2+q_2^\prime q_2^{\prime \prime }\right)  \left( B_3	u^2 \left(B_2u^2+q_2^\prime q_2^{\prime \prime }\right)^4+ \overline{ q_2^{\prime \prime }q_2^{\prime }} \right)  \equiv  \pm 1  \; {\rm mod} \; p. 
\end{align}
By the change of variable  $B_2u^2+q_2^\prime q_2^{\prime \prime } \mapsto u_3$, we arrive at 
 \begin{align} 
	h_4(u_3):=   B_3\overline{B_2}u_3^6-B_3\overline{B_2}
	q_2^{\prime }q_2^{\prime \prime} u_3^5+ \overline{ q_2^{\prime \prime }q_2^{\prime }} u_3 \mp 1 \equiv  0  \; {\rm mod} \; p. 
\end{align}
Let's  consider the negative sign (similar arguments hold true for the positive sign). In this case, we have 
 \begin{align*} 
	h_4(u_3)=(q_2^{\prime \prime }q_2^{\prime }B_3\overline{B_2}u_3^5+1)(\overline{ q_2^{\prime \prime }q_2^{\prime }}u_3-1)   \equiv  0  \; {\rm mod} \; p.
\end{align*}
Thus either $u_3 \equiv q_2^{\prime \prime}q_2^\prime  \; {\rm mod} \; p $ or $(q_2^{\prime \prime }q_2^{\prime }B_3\overline{B_2}u_3^5+1)  \equiv  0  \; {\rm mod} \; p$.  Thus, $h_4(u_3)$ has at most 6 solutions modulo $p$. Let $u_0$  be a solution of $h_4(u_3)$ modulo $p$. If $(\frac{\mathrm{d}h_3}{\mathrm{d}u_3}(u_0),p)=1$, then, using the Hensel's lemma, we get a unique lift. Moreover, it has  $p$-many lifts if   $\frac{\mathrm{d}h_3}{\mathrm{d}u_3}(u_0) \equiv\, 0\, \mathrm{mod} \, p
$ and $h_3(u_0) \equiv 0\, \mathrm{mod}\, p^2$. Let's first take $u_0=q_2^{\prime \prime}q_2^\prime$. Thus we get 
$$\frac{\mathrm{d}h_3}{\mathrm{d}u_3}(u_0) \equiv q_2^{\prime \prime }q_2^{\prime }B_3\overline{B_2}(q_2^{\prime \prime}q_2^\prime)^5+1=m\overline{m^\prime}q_2^{\prime 2}q_2^{\prime \prime 7}+1 \equiv 0\, \mathrm{mod}\, p. $$
Thus $m^\prime \equiv -mq_2^{\prime 2}q_2^{\prime \prime 7} \, \mathrm{mod}\, p$, from which we save $p$ which analyzing the sum over  $m^\prime$. This tells us that,if  we loose $p$  while lifting modulo $p^2$, it is gained back from the $m^\prime$ sum. Similar arguments holds for other solutions as well. Hence, on applying the Hensel lemma repeatedly, we get the desired claim \eqref{counting equation}.
Thus we have 
 \begin{align*}
	\mathfrak{C}_{\neq 0}(...) \ll p^{r+\ell_3+3(\ell-\ell_1)/2+\epsilon} \mathop{\sum \sum}_{\substack{d \vert q^\prime \\ d^\prime \vert q^{\prime \prime}}} dd^\prime\, \mathop{\sideset{}{^\star} \sum_{\alpha_1 \; {\rm mod} \; qk^\prime/n_1^\prime } \  \sideset{}{^\star} \sum_{\alpha_1^\prime \; {\rm mod} \; q^{\prime \prime }k^\prime/n_1^\prime} }_{\substack{ h_1(\alpha_1, \, m)\equiv \; 0 \; {\rm mod} \; d \\h_1(\alpha_1^\prime, \, m^\prime) \equiv \; 0 \; {\rm mod} \; d^{\prime} \\ \pm \overline{\alpha_1}q_2^{\prime \prime} \mp  \overline{\alpha_1^\prime}q_2^{ \prime}+n_2 \equiv\, 0 \, \mathrm{mod} \,q_1^\prime q_2^\prime q_2^{\prime \prime}k^\prime /n_1^\prime }}1, 
\end{align*}
 which we will analyze now.  The above sum can be dominated by a product of two sums $	\mathfrak{C}_{\neq 0}(...) \ll \mathfrak{C}_{\neq 0, 1}\mathfrak{C}_{\neq 0,2}$, where  
 \begin{align*}
 		\mathfrak{C}_{\neq 0,1} \ll p^{r+\ell_3+3(\ell-\ell_1)/2+\epsilon} \mathop{\sum \sum}_{\substack{d_1,d_1^\prime  \vert q_1^\prime }} d_1d_1^\prime\, \mathop{\sideset{}{^\star} \sum_{\alpha_1 \; {\rm mod} \; q_1^\prime k^\prime/n_1^\prime } \  \sideset{}{^\star} \sum_{\alpha_1^\prime \; {\rm mod} \; q_1^{ \prime }k^\prime/n_1^\prime} }_{\substack{ h_1(\alpha_1, \, m)\equiv \; 0 \; {\rm mod} \; d_1 \\h_1(\alpha_1^\prime, \, m^\prime) \equiv \; 0 \; {\rm mod} \; d_1^{\prime} \\ \pm \overline{\alpha_1}q_2^{\prime \prime} \mp  \overline{\alpha_1^\prime}q_2^{ \prime}+n_2 \equiv\, 0 \, \mathrm{mod} \,q_1^\prime k^\prime /n_1^\prime }}1,
 \end{align*}
  and 
   \begin{align*}
  	\mathfrak{C}_{\neq 0,2} \ll  \mathop{\sum \sum}_{\substack{d_2  \vert q_2^\prime \\ d_2^\prime \vert q_2^{\prime \prime } }} d_2d_2^\prime\, \mathop{\sideset{}{^\star} \sum_{\alpha_1 \; {\rm mod} \;  q_2^\prime } \  \sideset{}{^\star} \sum_{\alpha_1^\prime \; {\rm mod} \; q_2^{ \prime \prime } }}_{\substack{ h_1(\alpha_1, \, m)\equiv \; 0 \; {\rm mod} \; d_2 \\h_1(\alpha_1^\prime, \, m^\prime) \equiv \; 0 \; {\rm mod} \; d_2^{\prime} \\ \pm \overline{\alpha_1}q_2^{\prime \prime} \mp  \overline{\alpha_1^\prime}q_2^{ \prime}+n_2 \equiv\, 0 \, \mathrm{mod} \,q_2^\prime q_2^{\prime \prime }  }}1.
  \end{align*}
 In the second sum, since $(n_1^\prime, q_2^\prime q_2^{\prime \prime})=1$, we get $ \alpha_1  \equiv  -m\overline{n_1^\prime} \overline{p^{2(r+\ell_1)}}p^{2\ell} \, \mathrm{mod}\, d_2$ and  $ \alpha_1^\prime  \equiv  -m^\prime \overline{n_1^\prime} \overline{p^{2(r+\ell_1)}}p^{2\ell} \, \mathrm{mod}\, d_2^\prime$. Now using the congruence modulo $q_2^\prime q_2^{\prime \prime}$, we infer that 
 \begin{align}\label{c2}
 		\mathfrak{C}_{\neq 0,2} \ll  \mathop{\sum \sum}_{\substack{d_2  \vert (q_2^\prime, -n_2mp^{2\ell} \pm q_2^{\prime \prime} n_1^\prime p^{2(r+\ell_1)} ) \\  d_2^\prime  \vert (q_2^{\prime \prime}, n_2m^\prime p^{2\ell} \pm q_2^{ \prime} n_1^\prime p^{2(r+\ell_1)} )  }}d_2d_2^\prime.
 \end{align}
 In the first sum $\mathfrak{C}_{\neq 0,2}$, the congruence condition modulo $q_1^{\prime}k^{\prime} /n_1^\prime$ determines $\alpha_1^\prime$ uniquely in terms of $\alpha_1$, and hence
  \begin{align}\label{c1}
 	 \mathop{\sum \sum}_{\substack{d_1,d_1^\prime  \vert q_1^\prime }} d_1d_1^\prime\, \mathop{ \sideset{}{^\star} \sum_{\alpha_1 \; {\rm mod} \; q_1^\prime k^\prime/n_1^\prime } \  }_{\substack{ h_1(\alpha_1, \, m)\equiv \; 0 \; {\rm mod} \; d_1   }}1 \ll \frac{{q_1^{\prime 2} } k^\prime (m,n_1^\prime) }{n_1^\prime},
 \end{align}
 as 
 $h_1(\alpha_1,m)=n_1^\prime \alpha_1 p^{2\ell_1}\overline{p^{2\ell}}  \equiv  -m \overline{p^{2r}} \, \mathrm{mod}\, d_1 $ has $(n_1^\prime, m)$ many solutions modulo $d_1$.  Finally combining  estimates from  \eqref{c2} and \eqref{c1}, we get the lemma.
 \end{proof}

	\section{Estimates for the integral $\mathcal{J}(...)$ }
In this section, we will analyze the integral transform  $\mathcal{I}(...)$ given in \eqref{integral after poisson}. We have the following lemma.

\begin{lemma}\label{bound for integral}
	Let $\mathcal{I}(...)$  be  as in \eqref{integral after poisson}. Then we have 
	\begin{align*}
		\mathcal{I}\left(...\right) \ll \frac{M_1}{p^{2r-2\ell_2-\ell}}\frac{C^2Q^{\epsilon}}{Q^2}. 
	\end{align*}
Moreover, if $$|n_2| \gg R^{\epsilon} \frac{Q}{C} \frac{\mathcal{Q}}{N^\prime}:=N_2,$$
then $\mathcal{I}(...) $ is negligibly small. 
\end{lemma}
\begin{proof} 
	Let's first  recall from \eqref{integral after poisson} that 
		\begin{align}\label{I total}
		\mathcal{I}(...)=\int_{\mathbb{R}} V(w) \mathcal{J}\left(n_1^2 N^\prime w, p^{\ell^\prime}q^\prime,m\right) \overline{\mathcal{J}\left(n_1^2 N^\prime w, p^{\ell^\prime}q^{\prime \prime},m^\prime \right)} \; e\left(\frac{-n_{2}N^\prime w}{\mathcal{Q}  }\right) \mathrm{d}w,
	\end{align} 
where 
	\begin{align*}
	\mathcal{J}\left(n_1^2N^\prime w,  p^{\ell^\prime}q^\prime,  m\right)&=\int_{\mathbb{R}}W_1(x)g(p^{\ell^\prime}q^\prime,x)I_{1}(n_{1}^2 N^\prime w,p^{\ell^\prime}q^\prime,x)I_{2}\left(p^{\ell^\prime}q^\prime,m,x\right)\mathrm{d}x,
\end{align*}
%
which we will analyze now. Let's write  $q=p^{\ell^\prime}q^\prime$ for simplicity.	On plugging the expressions of $I_{1}(n_1^2N^\prime w,q,x)$ and $	I_2(q,m,x)$ from \eqref{integral of gl3} and \eqref{integral of gl2 } respectively in the above expression, we arrive at  the following expression of $	\mathcal{J}\left(n_1^2N^\prime w,  q,  m\right)$:
	\begin{align}\label{J int}
	\frac{1}{2\pi}\int_{-\infty}^{\infty } \int_{0}^{\infty } U(y)\int_{0}^{\infty }&\frac{W(z)}{\sqrt{z}}	\int_{\mathbb{R}}W_1(x)g(q,x)e\left(\frac{Nx(z-y)}{p^{\ell}qQ}\right) \gamma_{\pm}(-1/2+i\tau) \notag \\
	& \times  z^{-i\tau} e\left( \pm\frac{2\sqrt{Nmy}}{p^{r-\ell_2} q} \right)\left(\frac{n_1^2N^\prime Nw}{(qp^{\ell-\ell_1})^3 k}\right)^{-i\tau}  \mathrm{d}x\,  \mathrm{d}z\,   \mathrm{d}y \, \mathrm{d}\tau.
	\end{align}
On differentiating the above expression  with respect to $w$, we see that 
$$ \frac{\partial ^j}{\partial w^j}	\mathcal{J}\left(n_1^2N^\prime w,  q,  m\right) \ll  \left(\frac{N}{p^{\ell}qQ}\right)^{j+1} \ll \left(\frac{Q}{C}\right)^{j+1},$$
as $|\tau| \asymp N|x|/(p^{\ell}qQ)$ from \eqref{size of tau}, and $\gamma_{\pm}(-1/2+i\tau) \ll 1$.
Thus, on applying integration by parts repeatedly on the $w$-integral in \eqref{I total}, we see that $\mathcal{I}(...)$  is negligibly small unless 
$$|n_2| \ll R^{\epsilon} \frac{Q}{C} \frac{\mathcal{Q}}{N^{\prime}}.$$
This gives the second part of the lemma.
On considering the $y$-integral in \eqref{J int}, we observe that $$\frac{N|x|}{p^{\ell}qQ} \asymp \frac{\sqrt{Nm}}{p^{r-\ell_2}q},$$
as otherwise, using the first derivative bound, Lemma \ref{derivative bound}, the $y$-integral $I_2(q,m,x)$ will be negligibly small. Hence, we can assume that  
\begin{align}\label{range of x}
	|x| \asymp \sqrt{m}/p^{r-\ell_2-\ell/2}.
\end{align} 
%
%
%
%
%
\subparagraph*{Case 1} $q,\, p^{\ell^\prime}q^{\prime \prime} \sim C \geq Q^{1-\epsilon}$. \\
In this case, we estimate $	\mathcal{J}\left(n_1^2N^\prime w,  q,  m\right)$ trivially.   In fact, we have 
$$	\mathcal{J}\left(n_1^2N^\prime w,  q,  m\right) \ll \int_{\mathbb{R}}|W_1(x)||g(q,x)| \mathrm{d}x \ll Q^{\epsilon}\sqrt{M_1}/p^{r-\ell_2-\ell/2},$$
where we used \eqref{g properties} and \eqref{range of x} for the second inequality.  On analyzing  $\mathcal{J}\left(n_1^2 N^\prime w, q^\prime,m^\prime \right)$ similarly and plugging the correspoding estimates into \eqref{integral after poisson}, we get 
$$\mathcal{I}(...) \ll M_1/p^{2r-2\ell_2-\ell}.$$
\subparagraph*{Case 2} $q, \, p^{\ell^\prime}q^{\prime \prime} \sim C \ll  Q^{1-\epsilon}.$\\
	Let's  consider the $x$-integral in \eqref{J int}
	\begin{align*}
		I_{z-y}:=\int_{\mathbb{R}}W_1(x)g(q,x)e\left(\frac{Nx(z-y)}{p^{\ell}qQ}\right)\mathrm{d}x. 
	\end{align*}
We will analyze it in two subcases. 
\subparagraph*{Subcase 2.1}$m  \sim M_1 \asymp M_0=R^{\epsilon}p^{2r-2\ell_2-\ell}.$\\
In this situation, we have  
$$	I_{z-y}=\int_{|x| \asymp 1}W_1(x)g(q,x)e\left(\frac{Nx(z-y)}{p^{\ell}qQ}\right)\mathrm{d}x. $$
Using the second property \eqref{g properties} of $g(q,x)$, we observe that $$ \frac{\partial ^j}{\partial x^j}g(q,x) \ll \frac{\log Q}{|x|^j} \min \left\lbrace \frac{Q}{q}, \frac{1}{|x|}\right\rbrace \ll Q^{\epsilon j}.$$
Thus, using integration by parts repeatedly, the above integral $I_{z-y}$ is  negligibly small  unless $|z-y| \ll Q^\epsilon q/Q$. 
\subparagraph*{Subcase 2.2}$m  \sim M_1 \ll  M_0^{1-\epsilon}.$\\
In this case, we have the following $x$-integral
$$	I_{z-y}=\int_{|x| \ll R^{-\epsilon} }W_1(x)g(q,x)e\left(\frac{Nx(z-y)}{p^{\ell}qQ}\right)\mathrm{d}x. $$
Using the first property (see \eqref{g properties}) of $g(q,x)$,  we observe that $$g(q,x)-1=O \left(\frac{Q}{q} \left(\frac{q}{Q}+|x|\right)^{B}\right) \ll R^{-2020}.$$ 
	Thus we can replace $g(q,x)$ by $1$ at the cost of a negligible error term so that we essentially have
	\begin{align*}
		\int_{ |x| \ll R^{-\epsilon}} W_1(x)e\left(\frac{Nx(z-y)}{p^{\ell}qQ}\right)\mathrm{d}x.
	\end{align*}
	Now using integration by parts, we observe that the above integral is negligibly small unless $ |z-y| \ll Q^\epsilon{q}/{Q}.$ Thus combining Subcases 2.1 and 2.2, we conclude that the $x$-integral $I_{z-y}$ is negligibly small unless  $ |z-y| \ll Q^\epsilon{q}/{Q}$.
Now  writing  $z=y+u$, with $|u| \ll  Q^\epsilon{q}/{Q}$, we get the following expression   for $	\mathcal{J}\left(n_1^2N^\prime w,  q,  m\right)$:
	\begin{align}\label{simplified integral }
	\frac{1}{2\pi}&\int_{0}^{\infty }I_u\int_{-\infty}^{\infty } \ \gamma_{\pm}(-1/2+i\tau)\left(\frac{n_1^2N^\prime Nw}{(qp^{\ell-\ell_1})^3 k}\right)^{-i\tau}  \notag \\
& \times  \int_{0}^{\infty } U(y)\frac{W(y+u)}{\sqrt{y+u}}	(y+u)^{-i\tau} e\left( \pm\frac{2\sqrt{Nmy}}{p^{r-\ell_2} q} \right)   \mathrm{d}y \, \mathrm{d}\tau \, \mathrm{d}u.
\end{align}
Next we  analyze the integral over $y$. To this end, we will employ the stationary  phase expansion, Lemma \ref{stationaryphase},  to it.  We first observe that $|u|\ll C/Q \ll Q^{-\epsilon}$. Hence, on writing 
$$(y+u)^{-i\tau}=e^{-i \tau \log y}e^{-i\tau \log(1+u/y)},$$
we note that $e^{-i\tau \log(1+u/y)}$ can be inserted into the weight function $U$, as $$ \frac{\partial ^j}{\partial y^j} e^{-i\tau \log(1+u/y)} \ll Q^{\epsilon j}. $$
Thus the $y$-integral in \eqref{simplified integral } looks like 
\begin{align}
	  I(\tau, u):=\int_{0}^{\infty } U_u(y)	y^{-i\tau} e\left( \pm\frac{2\sqrt{Nmy}}{p^{r-\ell_2} q} \right)   \mathrm{d}y,
\end{align}
where $U_u(y)$ is the new weight function  $$U_u(y):= 2yU(y)W(y+u){(y+u)^{-1/2}} e^{-i\tau \log(1+u/y)}.$$  On taking $+$ sign, and using the change of variable $y \mapsto y^2$, we see that that the stationary point of the phase function is given by $y_0=\frac{ p^{r-\ell_2}q\tau}{2\pi \sqrt{Nm} }$. Thus on applying Lemma \ref{stationaryphase}, we arrive at 
  \begin{align*}
  	  I(\tau, u)=\frac{U_y(y_0)e(\tau\log(e/y_0) /\pi+1/8 )}{\sqrt{\tau/(\pi y_0^2)}}+\mathrm{lower \ order \  terms. }
  \end{align*}
On plugging the above expression in \eqref{simplified integral } and proceeding with the main term, we arrive at 
	\begin{align*}
	\frac{e(1/8)}{2\sqrt{\pi}}&\int_{0}^{\infty }I_u\int_{-\infty}^{\infty }  \gamma_{\pm}(-1/2+i\tau) \frac{y_0U_y(y_0)e(\tau\log(e/y_0) /\pi )}{\sqrt{\tau}}\left(\frac{n_1^2N^\prime Nw}{(qp^{\ell-\ell_1})^3 k}\right)^{-i\tau}   \mathrm{d}\tau \, \mathrm{d}u.
\end{align*}
On using the following  expansion  (due to Stirling formula) (see \cite{RM3})
\begin{align*}
	\gamma_{\pm}(-1/2+i\tau)=e^{3i\tau \log({\tau}/{e\pi})}\Phi_{\pm}(\tau), \ \  \Phi_{\pm}^{(j)}(\tau) \ll 1/\tau^j
\end{align*}
and applying the second derivative bound, Lemma \ref{derivative bound}, on the  $\tau$-integral , we see that it  is bounded by $1$. Hence 
	\begin{align*}
\mathcal{J}\left(n_1^2N^\prime w,  q,  m\right) \ll \int_{0}^{\infty } |I_u| \mathrm{d}u \ll  \frac{\sqrt{M_1}}{p^{r-\ell_2-\ell/2}}\frac{CQ^{\epsilon}}{Q},
\end{align*}
where we used  $$|I_u| \ll \sqrt{M_1}/p^{r-\ell_2-\ell/2},$$  
which follows using \eqref{range of x} and \eqref{g properties}. On analyzing $\mathcal{J}\left(n_1^2 N^\prime w, q^\prime,m^\prime \right)$ in a similar fashion, we get the first part of the lemma. 

\end{proof}
	\section{Final estimates for $\Omega$ and $S_k(N)$}\label{omega zero and nonzero }
In this section, we will estimate $ \Omega$ given in  \eqref{final Omega}. 
We will analyze it in two cases. 
	\subsection{The zero frequency} 
Let $\Omega_0$  denote the contribution of $n_2=0$ to $\Omega$, and let $S_{k,0}(N)$ be its contribution to $S_k(N)$ in \eqref{SN after cauchy}. 
	\begin{lemma}\label{zero frequeny bound}
	 For $\ell^\prime =0$, we have 
		\begin{align*}
			\Omega_0 \ll  \frac{ N_0p^{r+2\ell}R^{\epsilon} C^4M_0^{1/2}k}{ n_1^2p^{2\ell_1}Qq_1^\prime} 
		\end{align*}
		and 
		\begin{align*}
			S_{k,0}(N)\ll R^{\epsilon}N^{1/2}p^{3r/4+3\ell/4}.
		\end{align*}
	\end{lemma}
	\begin{proof}
Let's recall  from \eqref{final Omega} that 
	\begin{align*}
	\Omega_0 \ll \frac{1}{ M_{1}^{1/2}} \sup_{N^\prime \ll N_0/n_1^2}N^{\prime} \mathop{\sum \sum}_{q_2^\prime, q_2^{\prime \prime} \sim C/q_1^\prime} \mathop{\sum \sum}_{m,m' \sim M_{1}} \left|\mathfrak{C}_0(...)\right| \left|\mathcal{I}(...) \right|.
\end{align*}
On plugging the bound for $\mathfrak{C}_0(...)$  from  Lemma \ref{final character bound},  the following bound for $\mathcal{I}(...)$
	\begin{align*}
	\mathcal{I}\left(...\right) \ll \frac{M_1}{p^{2r-2\ell_2-\ell}}\frac{C^2Q^{\epsilon}}{Q^2} \ll \frac{C^2Q^{\epsilon}}{Q^2} 
\end{align*}
from  Lemma \ref{bound for integral}, and using the fact $q_2^{\prime \prime}=q_2^\prime$ in the above expression, we get 
\begin{align*}
		\Omega_0 \ll \frac{p^{r+2\ell}R^{\epsilon} C^2N_0}{ p^{2\ell_1}M_{1}^{1/2}Q^2n_1^2}   \mathop{\sum }_{q_2^\prime  \sim C/q_1^\prime}q k \mathop{\sum \sum}_{\substack{d ,d^{\prime} \vert q }}(d,d^{\prime})  \mathop{\sum \sum}_{\substack{m,m^\prime \sim M_{1} \\ (d,d^{\prime}) \vert \frac{(m-m^{\prime})}{p^{(r-\ell+\ell_1)}} }}1.
\end{align*}
Now  estimating the sum over $m$ and $m^\prime$, we arrive at
\begin{align*}
	\Omega_0 &\ll \frac{ N_0p^{r+2\ell}R^{\epsilon} C^2}{ n_1^2p^{2\ell_1}M_{1}^{1/2}Q^2}   \mathop{\sum }_{q_2^\prime  \sim C/q_1^\prime}q k \mathop{\sum \sum}_{\substack{d ,d^{\prime} \vert q }}\left(M_1(d,d^{\prime})+\frac{M_1^2}{p^{(r-\ell+\ell_1)}}\right) \\
	& \ll \frac{ N_0p^{r+2\ell}R^{\epsilon} C^2M_1^{1/2}}{ n_1^2p^{2\ell_1}Q^2}   \mathop{\sum }_{q_2^\prime  \sim C/q_1^\prime}q k  \left(q+\frac{M_1}{p^{(r-\ell+\ell_1)}}\right) \\
	& \ll \frac{ N_0p^{r+2\ell}R^{\epsilon} C^2M_0^{1/2}}{ n_1^2p^{2\ell_1}Q^2}   \frac{C^2k}{q_1^\prime}  \left(Q+\frac{M_0}{p^{(r-\ell+\ell_1)}}\right)  \\
		& \ll \frac{ N_0p^{r+2\ell}R^{\epsilon} C^2M_0^{1/2}}{ n_1^2p^{2\ell_1}Q^2}  \frac{C^2Qk}{q_1^\prime},
	\end{align*}
as $ M_0/p^{(r-\ell+\ell_1)} = R^\epsilon p^{r-\ell_1}  \ll R^{\epsilon}p^{3r/2-\ell/2}=Q$. Hence we have the first part of the lemma. To prove the second part, we substitute the above expression of $\Omega_0$ in \eqref{SN after cauchy}. Thus we see that  $S_{k,0}(N)$ is dominated by
	\begin{align*}
	 &   \mathop{\sup}_{\substack{C \ll Q \\ M_1 \ll M_0}} 	\frac{p^{\ell_1/2}N^{5/4}}{ Q p^{3\ell/2+r} k^{1/2}C^{2}}   \sum_{\frac{n_1^\prime}{(n_1^\prime ,\, k^\prime)} \ll C }  \Theta^{1/2} \sum_{\frac{n_1^\prime}{(n_1^\prime, \, k^{\prime})}\vert q_1^\prime \vert (n_1^\prime k^\prime)^\infty} \left(\frac{ N_0p^{r+2\ell}R^{\epsilon} C^4M_0^{1/2}k}{ n_1^2p^{2\ell_1}Q q_1^\prime}\right)^{1/2} \\
	   & \ll \mathop{\sup}_{\substack{C \ll Q }}\frac{p^{\ell_1/2}N^{5/4}}{ Q p^{3\ell/2+r} k^{1/2}}   \left(\frac{ N_0p^{r+2\ell} M_0^{1/2}k}{ p^{2\ell_1}Q }\right)^{1/2} \sum_{\frac{n_1^\prime}{(n_1^\prime ,\, k^\prime)} \ll C }  \Theta^{1/2} \sum_{\frac{n_1^\prime}{(n_1^\prime, \, k^{\prime})}\vert q_1^\prime \vert (n_1^\prime k^\prime)^\infty} \frac{1}{n_1\sqrt{q_1^\prime}} \\
	     & \ll  \mathop{\sup}_{\substack{C \ll Q }}\frac{p^{\ell_1/2}N^{5/4}}{ Q p^{3\ell/2+r} k^{1/2}}   \left(\frac{ N_0p^{r+2\ell} M_0^{1/2}k}{ p^{2\ell_1}Q }\right)^{1/2}	\sum_{n_1^\prime\ll Ck^\prime}\frac{(n_1^\prime ,k^\prime)^{1/2}}{n_1^{ \prime 3/2}}\Theta^{1/2},
\end{align*}
as $n_1=n_1^\prime$. Note that 
		\begin{align}\label{theta bound}
			\sum_{n_1^\prime \ll Ck^\prime }\frac{(n_1^\prime ,k^\prime )^{1/2} \Theta^{1/2} }{n_1^{ \prime 3/2}}\ll \left[\sum_{n_1^\prime \ll Ck^\prime}\frac{(n_1^\prime ,k^\prime)}{n_1^\prime}\right]^{1/2}\left[\mathop{\sum \sum}_{n_{1}^{ \prime 2} n_{2} \leq N_0} \frac{\vert A(n_{1}^\prime ,n_{2})\vert ^{2} }{n_1^{\prime 2}n_2}\right]^{1/2}\ll R^{\epsilon}.
		\end{align}
On using this bound, we arrive at
\begin{align*}
S_{k,0}(N) \ll \frac{p^{\ell_1/2}N^{5/4}}{ Q p^{3\ell/2+r} k^{1/2}}   \left(\frac{ N_0p^{r+2\ell} M_0^{1/2}k}{ p^{2\ell_1}Q }\right)^{1/2}\ll R^\epsilon N^{1/2}p^{3r/4+3\ell/4}.
\end{align*}
where we used $M_0=R^{\epsilon}p^{2r-\ell}$, $N_0=R^{\epsilon}\sqrt{N}p^{3\ell/2-3\ell_1}k$, $Q=(N/p^{\ell})^{1/2}$ and $Nk^2 \ll p^{3r}$.
Hence we have the lemma. 
	\end{proof}
	\subsection{The non-zero frequencies}
	Now it remains to   estimate $\Omega$ for non-zero values of $n_2$. Let $\Omega_{\neq 0}$  denote the contribution of $n_2 \neq 0$ to $\Omega$ in  \eqref{final Omega}, and let $S_{k,\neq 0}(N)$ be its contribution to $S_k(N)$ in \eqref{SN after cauchy}.  
	\begin{lemma}\label{omega bound for non-zero n_2}
		For $\ell^\prime =0$, we have 
	$$\Omega_{\neq 0}  \ll \frac{C^5k^2 p^{\ell-5\ell_1/2}p^{4r}}{ n_1 Q q_1^{\prime ^2}},$$
	and 
	$$S_{k,\neq 0}(N) \ll R^{\epsilon}{p^{r-\ell/2}kN^{3/4}}.$$
	\end{lemma}
	\begin{proof}
Let's recall from \eqref{final Omega} that 
\begin{align*}
	\Omega_{\neq 0}  \ll \frac{ 1 }{ M_{1}^{1/2}} \sup_{N^\prime \ll N_0/n_1^2} N^\prime \mathop{\sum \sum}_{q_2^\prime, q_2^{\prime \prime} \sim C/{q_1^\prime}} \mathop{\sum \sum}_{m,m^\prime \sim M_{1}} \sum_{0<|n_{2}| \ll N_2} \left|\mathfrak{C}_{\neq 0}(...)\right| \left|\mathcal{I}(...) \right|.
\end{align*}

$$ \frac{{q_1^{\prime  2} } k (m,n_1^\prime) }{n_1^\prime}p^{r+3(\ell-\ell_1)/2}$$
On plugging bounds for $\mathfrak{C}_{\neq 0}(...)$ and $\mathcal{I}(...)$ from  Lemma \ref{final character bound} and Lemma \ref{bound for integral} respectively  in the above expression, we arrive at
\begin{align*}
\frac{C^2q_1^{\prime 2}k p^{3(\ell-\ell_1)/2}M_1^{1/2}}{ p^{r-\ell} n_1 Q^2} \sup_{N^\prime \ll \frac{N_0}{n_1^2}} N^{\prime }   \mathop{\sum \sum}_{q_2^\prime, q_2^{\prime \prime} \sim C/{q_1^\prime}} \mathop{\sum \sum}_{\substack{d_2  \vert q_2^\prime \\ d_2^\prime \vert q_2^{\prime \prime } }} d_2d_2^\prime    \mathop{ \mathop{\sum \ \sum \ \  \ \sum}_{m, \, m^{\prime} \sim M_{1} \ \ 0< |n_2| \ll N_2 }}_{\substack{ -n_2m \pm q_2^{\prime \prime} n_1^\prime p^{\ell_6}  \, \equiv \,  0 \, {\rm mod} \, d_{2} \\   n_2m^\prime  \pm q_2^{ \prime} n_1^\prime p^{\ell_6} \, \equiv \,  0 \, {\rm mod} \, d_{2}^{\prime}}}  (m,n_1^\prime).
\end{align*}
By the change of variable $q_2^\prime \mapsto q_2^\prime d_2 $ and $q_2^{\prime \prime } \mapsto q_2^{\prime \prime }d_2^\prime $, we arrive at 
\begin{align*}
\frac{C^2q_1^{\prime 2}k p^{3(\ell-\ell_1)/2}M_1^{1/2}}{ p^{r-\ell} n_1 Q^2}  \sup_{N^\prime \ll \frac{N_0}{n_1^2}} N^{\prime 2/3}   \mathop{\sum \sum}_{d_{2},  d_{2}^{\prime} \ll \frac{C}{q_1^\prime} }d_2d_2^\prime  \mathop{\sum \sum }_{\substack{q_{2}^\prime \sim \frac{C}{d_{2}q_{1}^\prime} \\ q_{2}^{\prime \prime } \sim \frac{C}{d_{2}^{\prime} q_{1}^\prime }}}    \mathop{ \mathop{\sum \ \sum \ \  \ \sum}_{m, \, m^{\prime} \sim M_{1} \ \ 0< |n_2| \ll N_2 }}_{\substack{ -n_2m \pm q_2^{\prime \prime}d_2^\prime n_1^\prime p^{\ell_6}  \, \equiv \,  0 \, {\rm mod} \, d_{2} \\   n_2m^\prime  \pm q_2^{ \prime}d_2 n_1^\prime p^{\ell_6} \, \equiv \,  0 \, {\rm mod} \, d_{2}^{\prime}}}  (m,n_1^\prime).
\end{align*}
Next, we  count the number of $m$ in the above expression as follows:
	\begin{align*}
		\sum_{\substack{m \sim M_{1} \\-n_2m \pm q_2^{\prime \prime}d_2^\prime n_1^\prime p^{\ell_6}   \equiv   0 \, {\rm mod} \, d_{2} }}(n_1^\prime,m) =\sum_{\delta|n_1^\prime}\delta \sum_{\substack{m \sim M_{1}/\delta \\ -n_2m \pm \bar{\delta}q_2^{\prime \prime}d_2^\prime n_1^\prime p^{\ell_6}  \, \equiv \,  0 \, {\rm mod} \, d_{2} }}1 \ll  (d_{2},n_{2})  \left(n_1^\prime+\frac{M_{1}}{d_{2}}\right) \notag,
	\end{align*} 
Recall that $(n_1^\prime, d_2)=1$. Counting the number of $m^\prime$ in a similar fashion we get that the number of  $(m,m^\prime)$ pairs   is dominated by 
	$$ O((d_{2}^{\prime},  q_{2}^\prime d_{2}n_1^\prime) \, (d_{2}, n_{2}) (n_1^\prime+{M_{1}}/{d_{2}}) (1+{M_{1}}/{d_{2}^{\prime}})).$$ 
It follows that the contribution of this to $\Omega_{\neq 0}$ is dominated by
	\begin{align*}
	\frac{C^2q_1^{\prime 2}k p^{3(\ell-\ell_1)/2}M_1^{1/2}}{ p^{r-\ell} n_1 Q^2} \sup_{N^\prime \ll \frac{N_0}{n_1^2}} N^{\prime }   \mathop{\sum \sum}_{d_{2},  d_{2}^{\prime} \ll \frac{C}{q_1^\prime} }d_2d_2^\prime  
		 \mathop{\sum \sum }_{\substack{q_{2}^\prime \sim {C}/{d_{2}q_{1}^\prime} \\ q_{2}^{\prime \prime } \sim {C}/{d_{2}^{\prime} q_{1}^\prime }}} \\
		 \times  \sum_{ 1\leq n_2 \ll N_2}(d_{2}^{\prime},  q_{2}^\prime d_{2}n_1^\prime) \, (d_{2}, n_{2}) \left(n_1^\prime+\frac{M_{1}}{d_{2}}\right) \left(1+\frac{M_{1}}{d_{2}^{\prime}}\right).
	\end{align*}
Summing over $n_2$ and $q_2^{\prime \prime}$ we arrive at 
	\begin{align*}
\frac{C^3q_1^{\prime }k p^{3(\ell-\ell_1)/2}M_1^{1/2}}{ p^{r-\ell} n_1 Q^2} \frac{kCQp^{\ell-\ell_1}}{q_1^\prime n_1}   \mathop{\sum \sum}_{d_{2},  d_{2}^{\prime} \ll \frac{C}{q_1^\prime} }d_2
	\mathop{\sum }_{\substack{q_{2}^\prime \sim \frac{C}{d_{2}q_{1}^\prime}}} 
	(d_{2}^{\prime},  q_{2}^\prime d_{2}n_1^\prime) \left(n_1^\prime+\frac{M_{1}}{d_{2}}\right) \left(1+\frac{M_{1}}{d_{2}^{\prime}}\right),
\end{align*}
where we used $$\sup_{N^\prime \ll \frac{N_0}{n_1^2}} N^{\prime }N_2 \ll N^\prime \frac{Q}{C} \frac{\mathcal{Q}}{N^\prime} \ll \frac{kCQp^{\ell-\ell_1}}{q_1^\prime n_1}.$$
Next summing over $d_2^\prime$ we get 
	\begin{align*}
\frac{C^3q_1^{\prime }k p^{3(\ell-\ell_1)/2}M_1^{1/2}}{ p^{r-\ell} n_1 Q^2} \frac{kCQp^{\ell-\ell_1}}{q_1^\prime n_1}   \mathop{\sum }_{d_{2} \ll \frac{C}{q_1^\prime} }d_2
	\mathop{\sum }_{\substack{q_{2}^\prime \sim \frac{C}{d_{2}q_{1}^\prime}}} 
	 \left(n_1^\prime+\frac{M_{1}}{d_{2}}\right) \left(\frac{C}{q_1^\prime}+{M_{1}}\right).
\end{align*}
Executing the remaining sums we get 
	\begin{align*}
\frac{C^3q_1^{\prime }k p^{3(\ell-\ell_1)/2}M_1^{1/2}}{ p^{r-\ell} n_1 Q^2} \frac{kCQp^{\ell-\ell_1}}{q_1^\prime n_1}   \frac{C}{q_1^\prime}
	\left(\frac{Cn_1^\prime}{q_1^\prime}+{M_{1}}\right) \left(\frac{C}{q_1^\prime}+{M_{1}}\right). 
\end{align*}
Now bounding $M_1$ by $M_0$  and using $C \ll Q \ll M_0$, we get 
	\begin{align*}
	\Omega_{\neq 0} &\ll \frac{C^3q_1^{\prime }k p^{3(\ell-\ell_1)/2}M_1^{1/2}}{ p^{r-\ell} n_1 Q^2} \frac{kCQp^{\ell-\ell_1}}{q_1^\prime n_1}  \frac{C}{q_1^\prime}
	\frac{M_0^2n_1^\prime}{q_1^\prime} \\
	&\ll \frac{C^5k^2 p^{5(\ell-\ell_1)/2}M_0^{5/2}}{ p^{r-\ell} n_1 Q q_1^{\prime ^2}}  \ll \frac{C^5k^2 p^{\ell-5\ell_1/2}p^{4r}}{ n_1 Q q_1^{\prime ^2}},
\end{align*}
where we used $M_0 \ll R^{\epsilon}p^{2r-\ell}$. On using the above bound in \eqref{SN after cauchy}, we arrive at
	\begin{align} \label{SN after non zero}
	S_k(N) \ll  &\mathop{\sup}_{\substack{C \ll Q }} 	\frac{p^{\ell_1/2}N^{5/4}}{ Q p^{3\ell/2+r} k^{1/2}C^{2}} \sum_{\frac{n_1^\prime}{(n_1^\prime ,\, k^\prime)} \ll C }  \Theta^{1/2} \sum_{\frac{n_1^\prime}{(n_1^\prime, \, k^{\prime})}\vert q_1^\prime \vert (n_1^\prime k^\prime)^\infty}\frac{C^{5/2}kp^{\ell/2-5\ell_1/4}p^{2r}}{\sqrt{n_1Q}q_1^\prime} \notag \\ 
	&\ll \frac{p^{r-\ell-3\ell_1/4}N^{5/4}k^{1/2}}{ {Q} }   \sum_{\frac{n_1^\prime}{(n_1^\prime ,\, k^\prime)} \ll C}  \Theta^{1/2} \sum_{\frac{n_1^\prime}{(n_1^\prime, \, k^{\prime})}\vert q_1^\prime \vert (n_1^\prime k^\prime)^\infty}\frac{1}{\sqrt{n_1}q_1^\prime} \notag \\
	&\ll  \frac{p^{r-\ell-3\ell_1/4}N^{5/4}k^{1/2}}{ {Q} }   \sum_{\frac{n_1^\prime}{(n_1^\prime ,\, k^\prime)} \ll C}  \frac{(n_1^\prime,k^\prime)}{n_1^{\prime 3/2}}\Theta^{1/2}.
	\end{align}
Note that 
	\begin{align}
	\sum_{\frac{n_1^\prime}{(n_1^\prime ,\, k^\prime)} \ll C} \frac{(n_1^\prime ,k^\prime ) \Theta^{1/2} }{n_1^{ \prime 3/2}} \ll \left[\sum_{\frac{n_1^\prime}{(n_1^\prime ,\, k^\prime)} \ll C} \frac{(n_1^\prime ,k^\prime)^2}{n_1^\prime}\right]^{1/2}\left[\mathop{\sum \sum}_{n_{1}^{ \prime 2} n_{2} \leq N_0} \frac{\vert A(n_{1}^\prime ,n_{2})\vert ^{2} }{(n_1^{\prime 2}n_2)}\right]^{1/2}\ll \sqrt{k^\prime}R^{\epsilon}.
\end{align}
On plugging the above bound in  \eqref{SN after non zero}, we get 
\begin{align}
	S_{k,\neq 0}(N) \ll  \frac{p^{r-\ell-3\ell_1/4}N^{5/4}}{ {Q} }   \ll  {p^{r-\ell/2}kN^{3/4}},  
\end{align}
where we used $Q=\sqrt{N/p^\ell}$. Hence the lemma follows. 
	\end{proof}
	\section{Conclusion: proof of Theorem \ref{r>2 thm}} In this section, we will conclude the proof of Theorem \ref{r>2 thm}.  
	Using  bounds from Lemma \ref{zero frequeny bound} and Lemma \ref{omega bound for non-zero n_2},  we see that 
	\begin{align*}
		S_k(N) \ll |S_{k,0}(N)|+|S_{k,\neq 0}(N)| \ll  R^\epsilon N^{1/2}p^{3r/4+3\ell/4}+R^\epsilon{k}N^{3/4}p^{r-\ell/2}.
	\end{align*}
	On dividing the above equation by $k\sqrt{N}$, we get
	\begin{align*}
		S_k(N)/(k\sqrt{N}) \ll R^\epsilon p^{3r/4+3\ell/4}+R^\epsilon N^{1/4}p^{r-\ell/2} \ll  R^\epsilon p^{3r/4+3\ell/4}+R^\epsilon p^{7r/4-\ell/2}.
	\end{align*}
	Optimizing the above bound by equating  the terms on the right side, we get
	$$p^{3r/4+3\ell/4}=p^{7r/4-\ell/2} \iff p^{5\ell/4}=p^{r} \iff \ell =[4r/5].$$
On	plugging this in \eqref{AFE}, we get
	\begin{align} 
		L \left( \frac{1}{2}, \pi \times f \times \chi \right) \ll R^{\epsilon}p^{3r/4+3r/5} \ll R^{\epsilon}R^{3/2-3/20}.
	\end{align} 
Hence  Theorem \ref{r>2 thm} follows.
\section{Appendix} \label{mass case}
In this section, we will give a rough sketch of the proof of sub-convexity of 
$$L(1/2, E_{\min} \times f\times \chi).$$
Following Lemma \ref{AF}, the  problem boils down to getting cancellations in the following sum:
	\begin{align}\label{d_3)}
S(N) = \mathop{\sum}_{n =1}^{\infty} \, d_3(n) \, \lambda_{f}(n)\chi(n) \, W\left(\frac{n}{N}\right). 
\end{align}
After applying DFI and congruence equation  trick (see Subsection \ref{dfi and congruence equation}), we arrive at (upto negligible error terms)
	\begin{align} \label{S(N) for d3}
	S_{}(N)=& \frac{1}{Q p^\ell} \int_{\mathbb{R}}W_1(x) \sum_{1\leq q \leq Q} \frac{g(q,x)}{q} \sideset{}{^\star} \sum_{a \, {\mathrm{mod}}\, q} \sum_{b \, {\mathrm{mod}}\, p^{\ell}} \notag \\
	& \times \sum_{n=1}^{\infty} d_3(n) e\left(\frac{(a+bq)n}{p^\ell q}\right) e\left(\frac{nx}{p^\ell q Q}\right) W\left(\frac{n}{N}\right) \notag \\
	& \times \sum_{m=1}^{\infty} \lambda_f(m) \chi(m) e \left(\frac{-(a+bq)m}{p^\ell q}\right) e\left(\frac{-mx}{p^\ell qQ} \right)U\left(\frac{m}{N}\right). 
\end{align}

Next, we apply summation formulae to the  above  $n$-sum and $m$-sum. We first recall Voronoi formula for $d_3(n)$.  Set 
$$\sigma_{0,0}(k_1,k_2)=\sum_{d_1|k_2}\mathop{\sum}_{\substack{{d_2d_1 | k_2} \\ (d_2,k_1)=1}}1.$$
	Let $g$ be a compactly supported smooth function on  $ (0, \infty )$ and $\tilde{g}(s) = \int_{0}^{\infty} g(x) x^{s-1} \mathrm{d}x$ be its Mellin transform. For $\ell= 0$ and $1$, we define
\begin{equation*}
	\gamma_{\ell}(s) :=  \frac{\pi^{-3s-\frac{3}{2}}}{2} \, \left( \frac{\Gamma\left(\frac{1+s+ \ell}{2}\right)}{\Gamma\left(\frac{-s+ \ell}{2}\right)}\right)^3.
\end{equation*}
Set $\gamma_{\pm}(s) = \gamma_{0}(s) \mp \gamma_{1}(s)$ and let 
\begin{align}\label{gl3 integral transform for d_3}
	G_{\pm}(y) = \frac{1}{2 \pi i} \int_{(\sigma)} y^{-s} \, \gamma_{\pm}(s) \, \tilde{g}(-s) \, \mathrm{d}s,
\end{align}
where $\sigma > -1$. With the aid of the above terminology, we now state the $GL(3)$ Voronoi summation formula in the following lemma:
\begin{lemma} \label{gl3voronoi for d_3}
	Let $g(x)$ and  $d_3(n)$ be as above. Let $a, \bar{a}, q \in \mathbb{Z}$ with $c \neq 0, (a,c)=1,$ and  $a\bar{a} \equiv 1(\mathrm{mod} \ q)$. Then we have
	\begin{align} \label{GL3-Voro for d-3}
		&\sum_{n=1}^{\infty} d_3(n) e\left(\frac{an}{c}\right) g(n) \notag  \\
		=&\frac{1}{2c^2}\widetilde{g}(1) \sum_{n_{1}|c} n_1\tau(n_1)P_2(n_1,c) S\left( \bar{a}, 0; \frac{c}{n_1}\right) \notag \\
			&+\frac{1}{2c^2}\widetilde{g}^\prime(1) \sum_{n_{1}|c} n_1\tau(n_1)P_1(n_1,c) S\left( \bar{a}, 0; \frac{c}{n_1}\right) \notag \\
				&+\frac{1}{4c^2}\widetilde{g}^{\prime \prime}(1) \sum_{n_{1}|c} n_1\tau(n_1) S\left( \bar{a}, 0; \frac{c}{n_1}\right)  \notag \\
					&+{c} \sum_{\pm} \sum_{n_{1}|c} \sum_{n_{2}=1}^{\infty}  \frac{1}{n_{1} n_{2}} \sum_{n_3|n_1}\sum_{n_4|\frac{n_1}{n_3}}\sigma_{0,0}\left(\frac{n_1}{n_3n_4}, n_2\right) 
				S\left( \bar{a}, \pm n_{2}; \frac{c}{n_1}\right) G_{\pm} \left(\frac{n_{1}^2 n_{2}}{c^3 }\right)
	\end{align} 
	where 
	$P_1(n_1,c)=\frac{5}{3}\log n_1 -3\log c+3 \gamma-\frac{1}{3\tau(n)}\sum_{d|n} \log d,$ where $\gamma$ is the Euler constant and $P_2(n_1,c)$ is also some polynomial similar to $P_1(n_1,c)$ in $\log n_1$ and $\log c$. 
\end{lemma}
\begin{proof}
	See \cite{Li*} for the proof. 
\end{proof}
On applying the above lemma to the $n$-sum and  the $GL(2)$ Voronoi formula to the $m$-sum in \eqref{S(N) for d3}, we arrive at
$$S(N)=S_{\mathrm{error}}(N)+S_{\mathrm{main,\,  1}}(N)+S_{\mathrm{main,\,  2}}(N)+S_{\mathrm{main,\,  3}}(N),$$
where $S_{\mathrm{error}}(N)$ is the expression of $S(N)$ (after the Voronoi formulae) corresponding to last line of  \eqref{GL3-Voro for d-3},  $S_{\mathrm{main}, \, j}(N)$ is the expression of $S(N)$ corresponding to $(j+1)$-th line of  \eqref{GL3-Voro for d-3} for $j=1,\, 2$ and $3$.  We observe that the analysis of $S_{\mathrm{error}}(N)$  is exactly similar to that of $S(N)$ in \eqref{SN before cauchy}. Thus,  we will analyze $S_{\mathrm{main,\,  j}}(N)$ only. Let's consider 
$S_{\mathrm{main,\,  1}}(N)$.  Let's assume $q \sim Q$ for simplicity (generic case). 
 After Voronoi formula, $n$-sum has transfered to 
 $$\frac{1}{2(p^{\ell} q)^2}\widetilde{g}(1) \sum_{n_{1}|p^{\ell }q} n_1\tau(n_1)P_2(n_1,p^{\ell} q) S\left( \overline{(a+bq)}, 0; \frac{p^\ell q}{n_1}\right). 
  $$
 Assumming square root cancellations (which we will get on average over $a$) in the Kloosterman sum, we see the the above sum is bounded by $N/(p^{\ell} q)$, as $\widetilde{g}(1) \ll N$.  Thus we save $p^{\ell}q$ over the trivial bound which is  $N$. Analysis of $GL(2)$ Voronoi formula will give us a saving of size $N/{(p^rq)}$ over the trivial bound $N$.  Moreover, on analysing the sum over $a$ and $b$ like before, we  save $\sqrt{q}\sqrt{p^\ell}.$ Thus, in total,  we have saved 
 $$p^\ell q \times \frac{N}{p^rq}\times \sqrt{qp^\ell}=\frac{N\sqrt{q}p^{3\ell/2}}{p^{r}}$$
 over the trivial bound $N^2$.  This is sufficient as long as 
 $$\frac{N\sqrt{q}p^{3\ell/2}}{p^{r}} >N \iff \frac{\sqrt{N}p^{3\ell}}{p^{\ell/2}} >p^{2r} \iff p^{5\ell/2}>p^{r/2}.$$
  By our choice of $\ell$ which is $\ell=4r/5$, we get the subconvexity. 
%
%
	{}
\end{document}